\documentclass[11pt]{amsart}

\usepackage{amssymb,amsmath,amsfonts,amstext,latexsym}
\usepackage[all]{xy}
\usepackage[colorlinks=true, urlcolor=rltblue, citecolor=drkgreen, linkcolor=drkred] {hyperref}
\usepackage{color}
\usepackage{pdfsync}
\usepackage{MnSymbol} %smybols for infinitary formulas
\usepackage{enumitem}
\definecolor{rltblue}{rgb}{0,0,0.4}
\definecolor{drkgreen}{rgb}{0,0.4,0}
\definecolor{drkred}{rgb}{0.5,0,0}
\newtheorem{thm}{Theorem}%[section]
\newtheorem{lemma}[thm]{Lemma}
\newtheorem{proposition}[thm]{Proposition}

\newtheorem{theorem}[thm]{Theorem}
\newtheorem{corollary}[thm]{Corollary}

\newtheorem*{claim*}{Claim}

\newtheorem*{subclaim*}{Subclaim}
\newcounter{cases}[thm]
\newtheorem{case}[cases]{Case}

\numberwithin{subcase}{cases}

\theoremstyle{definition}
\newtheorem{definition}[thm]{Definition}
\newtheorem*{openquestion}{Open Question}
\newtheorem*{question}{Question}

\theoremstyle{remark}

\newtheorem{remark}[thm]{Remark}

\newtheorem{historic}[thm]{Historic Remark}

\theoremstyle{plain}

\newcounter{contenumi}

\def\la{\langle}
\def\ra{\rangle}
\def\and{\mathrel{\&}}

\def\ctt{{\mathtt c}}
\def\itt{{\mathtt {in}}}

\def\bfSigma{{\bf \Sigma}}

\def\On{On}

\newcounter{been}

\newcounter{eebeen}

\newcommand{\comment}[1]{}

\newcommand{\mc}[1]{\mathcal{#1}}

\DeclareMathOperator{\at}{at}
\DeclareMathOperator{\scs}{SS}

\DeclareMathOperator{\wf}{wfp}
\DeclareMathOperator{\wfc}{wfc}
\newcommand{\infi}[0]{\mathtt{in}}
\newcommand{\comp}[0]{\mathtt{c}}
\DeclareMathOperator{\sh}{sh}

\title{Scott Ranks of Models of a Theory}

\author[M. Harrison-Trainor]{Matthew Harrison-Trainor}
\address{Group in Logic and the Methodology of Science\\
University of California, Berkeley\\
 USA}
\email{matthew.h-t@math.berkeley.edu}
\urladdr{\href{http://math.berkeley.edu/~mattht}{http://math.berkeley.edu/$\sim$mattht}}

\thanks{The author was partially supported by the Berkeley Fellowship and NSERC grant PGSD3-454386-2014.
The author would like to thank Julia Knight, Theodore Slaman, and particularly Antonio Montalb\'an for his suggestions regarding Projective Determinacy and stationary sets in Section \ref{sec:scs}. }

\begin{document}

\begin{abstract}
The Scott rank of a countable structure is a measure, coming from the proof of Scott's isomorphism theorem, of the complexity of that structure. The Scott spectrum of a theory (by which we mean a sentence of $\mc{L}_{\omega_1 \omega}$) is the set of Scott ranks of countable models of that theory. In $ZFC + PD$ we give a descriptive-set-theoretic classification of the sets of ordinals which are the Scott spectrum of a theory: they are particular $\bfSigma^1_1$ classes of ordinals.

Our investigation of Scott spectra leads to the resolution (in $ZFC$) of a number of open problems about Scott ranks. We answer a question of Montalb\'an by showing, for each $\alpha < \omega_1$, that there is a $\Pi^{\infi}_2$ theory with no models of Scott rank less than $\alpha$. We also answer a question of Knight and Calvert by showing that there are computable models of high Scott rank which are not computably approximable by models of low Scott rank. Finally, we answer a question of Sacks and Marker by showing that $\delta^1_2$ is the least ordinal $\alpha$ such that if the models of a computable theory $T$ have Scott rank bounded below $\omega_1$, then their Scott ranks are bounded below $\alpha$.
\end{abstract}

\maketitle

\section{Introduction}
Scott \cite{Scott65} showed that every countable structure $\mc{A}$ can be characterized, up to isomorphism, as the the unique countable structure satisfying a particular sentence of the infinitary logic $\mc{L}_{\omega_1 \omega}$, called the \textit{Scott sentence} of $\mc{A}$. Scott's proof gives rise to a notion of \textit{Scott rank} for structures; there are several different definitions, which we will discuss later in Section \ref{ss:ranks}, but until then we may take the Scott rank of $\mc{M}$ to be the least ordinal $\alpha$ such that $\mc{M}$ has a $\Pi^{\infi}_{\alpha+1}$ Scott sentence. This paper is concerned with the following general question: given a theory (by which we mean a sentence of $\mc{L}_{\omega_1 \omega}$) what could the Scott ranks of models of $T$ be? This collection of Scott ranks is called the \textit{Scott spectrum} of $T$:
\begin{definition}
Let $T$ be an $\mc{L}_{\omega_1 \omega}$-sentence. The \textit{Scott spectrum} of $T$ is the set
\[ \scs(T) = \{ \alpha \in \omega_1 \colon \alpha \text{ is the Scott rank of a countable model of } T \}.\]
\end{definition}
\noindent This is an old definition. For example, in 1981, Makkai \cite{Makkai81} defined the Scott spectrum of a theory in this way and showed that there is a sentence of $\mc{L}_{\omega_1 \omega}$ without uncountable models whose Scott spectrum is unbounded below $\omega_1$. In \cite[p. 151]{Jouko11} a reference is made to \textit{gaps} in the Scott spectrum---ordinals $\beta$ which are not in the Scott spectrum, but which are bounded above by some other $\alpha$ in the Scott spectrum---but the only results proved about Scott spectra are about bounds below $\omega_1$. This seems to be a general pattern: whenever Scott spectra are mentioned in the literature, it is to say that they are either bounded or unbounded below $\omega_1$. This paper, to the contrary, is about the gaps, and about a classification of the sets of countable ordinals that can be Scott spectra. Our main result is a complete descriptive-set-theoretic classification of the sets of ordinals which are Scott spectra. For this classification, we assume projective determinacy.

This work began with the following question, first asked by Montalb\'an at the 2013 BIRS Workshop on Computable Model Theory.
\begin{question}[Montalb\'an]
If $T$ is a $\Pi^{\infi}_2$ sentence, must $T$ have a model of Scott rank two or less?
\end{question}
\noindent At the time, we knew very little about how to answer such questions. In this paper, we make a large step forward in our understanding of Scott spectra: not only do we answer the question negatively, but we also answer the generalization to any ordinal $\alpha$ and we apply those techniques to solve other open problems about Scott ranks.

The paper is in two parts. The first part is a general construction in Section \ref{sec:gen-const}. Given a $\mc{L}_{\omega_1 \omega}$-pseudo-elementary class of linear orders, we build an $\mc{L}_{\omega_1 \omega}$-sentence $T$ so that the Scott spectrum of $T$ is related to the set of well-founded parts of linear orders in that class. The construction bears some similarity to work of Marker \cite{Marker90}. In the second part, we apply the general construction to get various results about Scott spectra. We will describe these applications now.

\subsection{\texorpdfstring{$\Pi^{\infi}_2$ theories with no models of low Scott rank}{Pi2 theories with no models of low Scott rank}}

It follows easily from known results that for a given ordinal $\alpha$, there is a theory $T$ all of whose models have Scott rank at least $\alpha$. (We can, for example, take $T$ to be the Scott sentence of a model of Scott rank $\alpha$.) This is not very surprising, as the theory $T$ we get has quantifier complexity about $\alpha$; complicated theories may have only complicated models. The interesting question is whether there is an uncomplicated theory all of whose models are complicated. Such theories exist.

\begin{theorem}\label{thm:pi-two}
Fix $\alpha < \omega_1$. There is a $\Pi^{\infi}_2$ sentence $T$ whose models all have Scott rank $\alpha$.
\end{theorem}

\noindent In particular, taking $\alpha > 2$ answers the question of Montalb\'an stated above. In Section \ref{sec:pi-two} we will derive Theorem \ref{thm:pi-two} from the general construction of Section \ref{sec:gen-const}.

\subsection{Computable structures of high Scott rank}

Nadel \cite{Nadel74} showed that if $\mc{A}$ is a computable structure, then its Scott rank is at most $\omega_1^{CK} + 1$. We say that a computable structure with non-computable Scott rank, i.e.\ with Scott rank $\omega_1^{CK}$ or $\omega_1^{CK} + 1$, has high Scott rank. There are few known examples of computable structures of high Scott rank. Harrison \cite{Harrison68} gave the first example of a structure of Scott rank $\omega_1^{CK} + 1$: the Harrison linear order $\mc{H}$, which is a computable linear order of order type $\omega_1^{CK} (1 + \mathbb{Q})$. The Harrison order is the limit of the computable ordinals in the following sense: given $\alpha$ a computable ordinal, there is a computable ordinal $\beta$ such that $\mc{H} \equiv_\alpha \beta$. We say that such a structure is strongly computable approximable:

\begin{definition}
A computable structure $\mc{A}$ of non-computable rank is \textit{weakly computably approximable} if every computable infinitary sentence $\varphi$ true in $\mc{A}$ is also true in some computable $\mc{B} \ncong \mc{A}$. $\mc{A}$ is \textit{strongly computably approximable} if we require that $\mc{B}$ have computable Scott rank.
\end{definition}

Makkai \cite{Makkai81} gave the first example of an arithmetic structure of Scott rank $\omega_1^{CK}$, and Knight and Millar \cite{KnightMillar10} modified the construction to get a computable structure. Calvert, Knight, and Millar \cite{CalvertKnightMillar06} showed that this structure is also strongly computably approximable. Calvert and Knight \cite[Problem 6.2]{CalvertKnight06} asked the following question:

\begin{question}[Calvert and Knight]
Is every computable model of high Scott rank strongly (or weakly) computably approximable?
\end{question}

At the time, every known example of a computable structure of high Scott rank was strongly computably approximable. We show here that there are computable structures of Scott rank $\omega_1^{CK}$ and $\omega_1^{CK} + 1$ which are not strongly computably approximable.

\begin{theorem}\label{thm:comp}
For $\alpha = \omega_1^{CK}$ or $\alpha = \omega_1^{CK} + 1$ : There is a computable model $\mc{A}$ of Scott rank $\alpha$ and a $\Pi^{\comp}_2$ sentence $\psi$ such that $\mc{A} \models \psi$, and whenever $\mc{B}$ is any structure and $\mc{B} \models \psi$, $\mc{B}$ has Scott rank $\alpha$.
\end{theorem}

\noindent We prove Theorem \ref{thm:comp} in Section \ref{sec:comp}. Note that this gives a new type of model of high Scott rank which is qualitatively different from the previously known examples.

\subsection{Bounds on Scott Height}

It follow from a general counting argument that there is a least ordinal $\alpha < \omega_1$ such that if $T$ is a computable $\mc{L}_{\omega_1 \omega}$-sentence whose Scott spectrum is bounded below $\omega_1$, then the Scott spectrum of $T$ is bounded below $\alpha$. We call this ordinal the Scott height of $\mc{L}^{\comp}_{\omega_1 \omega}$, and we denote it $\sh(\mc{L}^{\comp}_{\omega_1,\omega})$.

Sacks \cite{Sacks83} and Marker \cite{Marker90} asked:
\begin{question}[Sacks and Marker]
What is the Scott height of $\mc{L}^{\comp}_{\omega_1 \omega}$?
\end{question}

\begin{definition}
$\delta^1_2$ is the least ordinal which has no $\Delta^1_2$ presentation.
\end{definition}

Sacks \cite{Sacks83} showed that $\sh(\mc{L}^{\comp}_{\omega_1,\omega}) \leq \delta^1_2$. Marker \cite{Marker90} was able to resolve this question for pseudo-elementary classes.

\begin{definition}
A class $\mc{C}$ of structures in a language $\mc{L}$ is an $\mc{L}_{\omega_1 \omega}$-pseudo-elementary class ($PC_{\mc{L}_{\omega_1 \omega}}$-class) if there is an $\mc{L}_{\omega_1 \omega}$-sentence $T$ in an expanded language $\mc{L}' \supseteq \mc{L}$ such that the structures in $\mc{C}$ are the reducts to $\mc{L}$ of the models of $T$. $\mc{C}$ is a computable $PC_{\mc{L}_{\omega_1 \omega}}$-class if $T$ is a computable sentence.
\end{definition}

\noindent We can define the Scott height of $PC_{\mc{L}_{\omega_1 \omega}}$ in a similar way to the Scott height of $\mc{L}^{\comp}_{\omega_1 \omega}$, except that now we consider all $\mc{L}_{\omega_1 \omega}$-pseudo-elementary classes which are the reducts of the models of a computable sentence. Marker \cite{Marker90} showed that $\sh(PC^{\comp}_{\mc{L}_{\omega_1 \omega}}) = \delta^1_2$. Using our methods, we can expand this argument to $\mc{L}^{\comp}_{\omega_1 \omega}$.

\begin{theorem}\label{thm:sh}
$\sh(\mc{L}^{\comp}_{\omega_1,\omega}) = \delta^1_2$.
\end{theorem}

We prove this theorem in Section \ref{sec:sh}.

\subsection{Classifying the Scott spectra}

Assuming projective determinacy, we will define a descriptive set-theoretic class which will give a classification of the Scott spectra.

\begin{definition}
A set of countable ordinals is a $\bfSigma^1_1$ class of ordinals if it consists of the order types in $\mc{C} \cap \On$ for some $\bfSigma^1_1$ class $\mc{C}$ of linear orders on $\omega$.
\end{definition}

\noindent Note that $\mc{C}$ and $\On$ here are classes of \textit{presentations} of ordinals as linear orders of $\omega$. Frequently we will pass without comment between viewing a class as a collection of ordinals, i.e., of order types, and as a collection of $\omega$-presentations of linear orders.

\begin{theorem}[ZFC + PD]\label{thm:main-result}
The Scott spectra of $\mc{L}_{\omega_1 \omega}$-sentences are the $\bfSigma^1_1$ classes $\mc{C}$ of ordinals with the property that, if $\mc{C}$ is unbounded below $\omega_1$, then either $\mc{C}$ is stationary or $\{ \alpha \colon \alpha + 1 \in \mc{C} \}$ is stationary.
\end{theorem}

We can also get an alternate characterization which is more tangible. To state this, we must define two ways to produce an ordinal from an arbitrary linear order.

\begin{definition}
Let $(L,\leq)$ be a linear order. The \textit{well-founded part} $\wf(L)$ of $L$ is the largest initial segment of $L$ which is well-founded. The \textit{well-founded collapse} of $L$, $\wfc(L)$, is the order type of $L$ after we collapse the non-well-founded part $L \setminus \wf(L)$ to a single element.
\end{definition}
\noindent We can identify $\alpha \in \wf(L)$ with the ordinal which is the order type of $\{\beta \in L : \beta < \alpha\}$. We can also identify $\wf(L)$ with its order type. If $L$ is well-founded, with order type $\alpha$, then $\wfc(L) = \wf(L) = \alpha$. If $L$ is not well-founded, $\wfc(L) = \wf(L)+1$.

\begin{theorem}[ZFC + PD]\label{thm:main-result2}
The Scott spectra of $\mc{L}_{\omega_1 \omega}$-sentences are exactly the sets of the form:
\begin{enumerate}
	\item $\wf(\mc{C})$,
	\item $\wfc(\mc{C})$, or
	\item $\wf(\mc{C}) \cup \wfc(\mc{C})$
\end{enumerate}
where $\mc{C}$ is a $\bfSigma^1_1$ class of linear orders of $\omega$.
\end{theorem}

\begin{theorem}[ZFC + PD]\label{thm:all-pi-two}
Each Scott spectrum is the Scott spectrum of a $\Pi^{\infi}_2$ sentence.
\end{theorem}

\begin{theorem}[ZFC + PD]\label{thm:main-result-pseudo}
Every Scott spectrum of a $PC_{\mc{L}_{\omega_1 \omega}}$-class is the Scott spectrum of an $\mc{L}_{\omega_1 \omega}$-sentence.
\end{theorem}

\noindent We will prove Theorems \ref{thm:main-result}, \ref{thm:main-result2}, \ref{thm:all-pi-two}, and \ref{thm:main-result-pseudo} in Section \ref{sec:scs}. This classification allows us to construct interesting Scott spectra. For example, the successor ordinals and the admissible ordinals are Scott spectra.

%We believe that, in light of Theorem \ref{thm:pi-two}, the general construction may be of independent interest. The majority of theories---for example, the theories of various algebraic structures such as algebraically closed fields---which occur in mathematics are $\Pi^{\infi}_2$ theories. Thus, when presenting a theory which is an example of some phenomenon, it is much more satisfying to exhibit a $\Pi^{\infi}_2$ theory; such an example is less ad hoc. The ideas of the general construction can be used to code very complex behaviour into a $\Pi^{\infi}_2$ theory, and may be useful for other applications.

%\subsection{A Lightface Classification for Computable Theories}
%
%We also get a corresponding lightface version of Theorem \ref{thm:main-result}.
%
%\begin{definition}
%A set of countable ordinals $\mathfrak{A}$ is a $\Sigma^1_1$-$\wfc$ class of ordinals if there is a $\Sigma^1_1$ set $\mc{C} \subseteq \mc{P}(\omega \times \omega)$ of linear orders on $\omega$, and
%\[ \mathfrak{A} = \{ \wfc(L) : L \in \mc{C} \}.\]
%\end{definition}
%
%\begin{theorem}\label{thm:main-result-eff}
%The Scott spectra of computable $\mc{L}_{\omega_1 \omega}$-sentences are the $\Sigma^1_1$-$\wfc$ classes of ordinals. Moreover, every $\Sigma^1_1$-$\wfc$ class of ordinals is the Scott spectrum of a $\Pi^{\comp}_2$ sentence.
%\end{theorem}

\section{Preliminaries on Back-and-forth Relations and Scott Ranks}
All of our structures will be countable structures in a countable language. The infinitary logic $\mc{L}_{\omega_1 \omega}$ consists of formulas which allow countably infinite conjunctions and conjunctions; see \cite[Sections 6 and 7]{AshKnight00} for background. We will use $\Sigma^{\itt}_\alpha$ for the infinitary $\Sigma_\alpha$ formulas and $\Sigma^{\ctt}_\alpha$ for the computable infinitary $\Sigma_\alpha$ formulas (and similarly for $\Pi^{\itt}_\alpha$ and $\Pi^{\ctt}_\alpha$).

\subsection{Scott Rank}\label{ss:ranks}

Let $\mc{A}$ be a countable structure. There are a number of ways to define the Scott rank of $\mc{A}$, not all of which agree. We describe a number of different definitions before fixing one for the rest of the paper. For the most part, it does not matter, modulo some small changes, which definition we choose as our results are quite robust.

The first definition uses the symmetric back-and-forth relations which come from Scott's proof of his isomorphism theorem \cite{Scott65}. See, for example, \cite[Sections 6.6 and 6.7]{AshKnight00}.
\begin{definition}
The \textit{standard symmetric back-and-forth relations} $\sim_\alpha$ on $\mc{A}$, for $\alpha < \omega_1$, are defined by:
\begin{enumerate}
	\item $\bar{a} \sim_0 \bar{b}$ if $\bar{a}$ and $\bar{b}$ satisfy the same quantifier-free formulas.
	\item For $\alpha > 0$, $\bar{a} \sim_\alpha \bar{b}$ if for each $\beta < \alpha$ and $\bar{d}$ there is $\bar{c}$ such that $\bar{a} \bar{c} \sim_\beta \bar{b} \bar{d}$, and for all $\bar{c}$ there is $\bar{d}$ such that $\bar{a} \bar{c} \sim_\beta \bar{b} \bar{d}$.
\end{enumerate}
\end{definition}

For each tuple $\bar{a} \in \mc{A}$, Scott proved that there is a least ordinal $\alpha$, the Scott rank of the tuple, such that if $\bar{a} \sim_\alpha \bar{b}$, then $\bar{a}$ and $\bar{b}$ are in the same automorphism orbit of $\mc{A}$. Equivalently, $\alpha$ is the least ordinal such that if $\bar{a} \sim_\alpha \bar{b}$, then $\bar{a} \sim_\gamma \bar{b}$ for all ordinals $\gamma < \omega_1$, or such that if $\bar{a} \sim_\alpha \bar{b}$, then $\bar{a}$ and $\bar{b}$ satisfy the same $\mc{L}_{\omega_1 \omega}$-formulas. Then the Scott rank of $\mc{A}$ is the least ordinal strictly greater than (or, in the definition used by Barwise \cite{Barwise75}, greater than or equal to) the Scott rank of each tuple of $\mc{A}$. One can then define a \textit{Scott sentence} for $\mc{A}$, that is, a sentence of $\mc{L}_{\omega_1 \omega}$ which characterizes $\mc{A}$ up to isomorphism among countable structures.

Another definition uses the non-symmetric back-and-forth relations which have been useful in computable structure theory. See \cite[Section 6.7]{AshKnight00}.
\begin{definition}\label{def:bf}
The \textit{standard (non-symmetric) back-and-forth relations} $\leq_\alpha$ on $\mc{A}$, for $\alpha < \omega_1$, are defined by:
\begin{enumerate}
	\item $\bar{a} \leq_0 \bar{b}$ if for each quantifier-free formula $\psi(\bar{x})$ with G\"odel number less than the length of $\bar{a}$, if $\mc{A} \models \psi(\bar{a})$ then $\mc{A} \models \psi(\bar{b})$.
	\item For $\alpha > 0$, $\bar{a} \leq_\alpha \bar{b}$ if for each $\beta < \alpha$ and $\bar{d}$ there is $\bar{c}$ such that $\bar{b} \bar{d} \leq_\beta \bar{a} \bar{c}$.
\end{enumerate}
Let $\bar{a} \equiv_\alpha \bar{b}$ if $\bar{a} \leq_\alpha \bar{b}$ and $\bar{b} \leq_\alpha \bar{a}$,
\end{definition}
\noindent For $\alpha \geq 1$, $\bar{a} \leq_\alpha \bar{b}$ if and only if every $\Sigma^{\infi}_1$ formula true of $\bar{b}$ is true of $\bar{a}$.

Then one can define the Scott rank of a tuple $\bar{a}$ to be the least $\alpha$ such that if $\bar{a} \equiv_\alpha \bar{b}$, then $\bar{a}$ and $\bar{b}$ are in the same automorphism orbit of $\mc{A}$. The Scott rank of $\mc{A}$ is then least ordinal strictly greater than the Scott rank of each tuple.

A third definition of Scott rank has recently been suggested by Montalb\'an based on the following theorem:
\begin{theorem}[Montalb\'an \cite{Montalban}]\label{thm:montalban}
Let $\mc{A}$ be a countable structure, and $\alpha$ a countable ordinal. The following are equivalent:
\begin{enumerate}
	\item $\mc{A}$ has a $\Pi^{\infi}_{\alpha+1}$ Scott sentence.
	\item Every automorphism orbit in $\mc{A}$ is $\Sigma^{\infi}_\alpha$-definable without parameters.
	\item $\mc{A}$ is uniformly (boldface) $\mathbf{\Delta}^0_\alpha$-categorical without parameters.
	\item Every $\Pi^{\infi}_\alpha$ type realized in $\mc{A}$ is implied by a $\Sigma^{\infi}_\alpha$ formula.
	\item\label{part:free} No tuple in $\mc{A}$ is $\alpha$-free.
\end{enumerate}
\end{theorem}
Montalb\'an defines the Scott rank of $\mc{A}$ to be the least ordinal $\alpha$ such that $\mc{A}$ has a $\Pi^{\infi}_{\alpha+1}$ Scott sentence. It is this definition which we will take as our definition of Scott rank. We write $SR(\mc{A})$ for the Scott rank of the structure $\mc{A}$. The $\alpha$-free tuples which appear in the theorem above will also appear later.
\begin{definition}
Let $\bar{a}$ be a tuple of $\mc{A}$. Then $\bar{a}$ is \textit{$\alpha$-free} if for each $\bar{b}$ and $\beta < \alpha$, there are $\bar{a}'$ and $\bar{b}'$ such that $\bar{a}, \bar{b} \leq_\beta \bar{a}', \bar{b}'$ and $\bar{a}' \nleq_\alpha \bar{a}$.
\end{definition}

Other definitions of Scott rank appear in \cite[Section 2]{Sacks07} and \cite[Section 3]{Gao07}.

\subsection{Scott Spectra}

Recall that the Scott spectrum of an $\mc{L}_{\omega_1 \omega}$-sentence $T$ is the set of countable ordinals
\[ \scs(T) = \{ SR(\mc{A}) : \mc{A} \text{ is a countable model of } T \}. \]
More generally, one can define the Scott spectrum $\scs(\mc{C})$ of a class of countable structures $\mc{C}$.
For each $\alpha < \omega_1$ there is an $\mc{L}_{\omega_1 \omega}$-sentence whose Scott spectrum is $\{\alpha\}$. For example, if $\mc{A}$ is a structure of Scott rank $\alpha$,\footnote{Such structures exist; for example, the results on linear orders in \cite[Section 15]{AshKnight00} can be used to construct examples, or one can use the construction in \cite{CFS}.} then we can take $T$ to be the Scott sentence for $\mc{A}$. However, the quantifier complexity of $T$ will be approximately $\alpha$. It is only as a result of our Theorem \ref{thm:pi-two} that one can obtain such a theory $T$ of low quantifier complexity even when $\alpha$ is very large.

We note some results about producing new Scott spectra by combining existing ones. The proofs are all simple constructions which we omit.

\begin{proposition}\label{prop:ctble-union}
If $\langle \mathfrak{A}_i \rangle_{i \in \omega}$ are the Scott spectra of $\mc{L}_{\omega_1 \omega}$-sentences, then $\bigcup_{i \in \omega} \mathfrak{A}_i$ is also the Scott spectrum of an $\mc{L}_{\omega_1 \omega}$-sentence.
\end{proposition}
%\begin{proof}
%For $i \in \omega$ let $T_i$ be a theory with $\scs(T_i) = \mathfrak{A}_i$ in a language $\mathcal{L}_i$. We will define a theory $S$ in the language $\{A,B\} \cup \{R_i : i \in \omega\} \cup \bigcup_{i \in \omega} \mathcal{L}_i$, where the new symbols are all unary relations. The theory $S$ has the following axioms:
%\begin{enumerate}
	%\item[(i)] there are two sorts $A$ and $B$ (i.e., the universe is partitioned into $A$ and $B$) and $A$ contains only one element,
	%\item[(ii)] exactly one of the $R_i$ holds of the single element of $A$,
	%\item[(iii)] if $R_i$ holds of the single element of $A$, then each relation in $\mathcal{L}_j$ for $j \neq i$ is trivial, and $B$ together with the interpretations of $\mathcal{L}_i$ is a model of $T_i$.
%\end{enumerate}
%A model of $S$ is essentially a model of $T_i$ for some $i$, and each model of $T_i$ is essentially a model of $S$, and the Scott ranks are the same.
%\end{proof}

\begin{proposition}\label{prop:cut-off}
If $\mathfrak{A}$ is the Scott spectrum of an $\mc{L}_{\omega_1 \omega}$-sentence and $\alpha < \omega_1$, then $\mathfrak{B} = \{\beta \in \mathfrak{A} : \beta \geq \alpha \}$ is also the Scott spectrum of an $\mc{L}_{\omega_1 \omega}$-sentence.
\end{proposition}
%\begin{proof}
%Let $T_1$ be a theory with $\scs(T_1) = \mathfrak{A}$ in a language $\mathcal{L}_1$. Let $T_2$ be the Scott sentence of a structure with Scott rank $\alpha$ in a language $\mc{L}_2$. We will define a theory $S$ in the language $\{A \cup B\} \cup \mathcal{L}_1 \cup \mathcal{L}_2$. The theory $S$ just says that $A$ and $B$ partition the domain into disjoint sorts, $\mathcal{L}_1$ is defined on $A$, $\mathcal{L}_2$ is defined on $B$, and $A$ is a model of $T_1$ and $B$ is a model of $T_2$. So the models of $S$ are essentially disjoint unions of models of $T_1$ and $T_2$. It is easy to see that the Scott rank of such a disjoint union is the maximum of the Scott ranks of the two models.
%\end{proof}

\begin{proposition}\label{prop:change-spec}
Let $\mathfrak{A}$ and $\mathfrak{B}$ be sets of countable ordinals, and suppose that $\mathfrak{A}$ is the Scott spectrum of an $\mc{L}_{\omega_1 \omega}$-sentence. If there is a countable ordinal $\alpha < \omega_1$ such that
\[ \mathfrak{A} \cap \{ \beta : \alpha \leq \beta < \omega_1\} = \mathfrak{B} \cap \{ \beta : \alpha \leq \beta < \omega_1\} \]
then $\mathfrak{B}$ is also the Scott spectrum of an $\mc{L}_{\omega_1 \omega}$-sentence.
\end{proposition}
%\begin{proof}
%We have
%\[ \mathfrak{B} = \{ \beta \in \mathfrak{A} : \beta \geq \alpha\} \cup \bigcup_{\mathfrak{B} \ni \beta < \alpha} \{ \beta\}.\]
%By Proposition \ref{prop:cut-off}, $\{ \beta \in \mathfrak{A} : \beta \geq \alpha\}$ is the Scott spectrum of a theory. For each $\beta < \alpha$, $\{\beta\}$ is also the Scott spectrum of a theory, so by Proposition \ref{prop:ctble-union}, $\mathfrak{B}$ is also the Scott spectrum of a theory.
%\end{proof}

\subsection{Non-standard Back-and-Forth Relations}

Let $(L,\leq)$ be a linear order. We will consider $(L,\leq)$ to be a \textit{non-standard ordinal}, i.e., a linear ordering with an initial segment which is an ordinal, but whose tail may not necessarily be well-ordered. Assume that $L$ has a smallest element $0$.

\begin{definition}\label{def:nsbf}
A sequence of equivalence relations $(\precsim_\alpha)_{\alpha \in L}$ are \textit{non-standard back-and-forth relations} on $\mc{A}$ if they satisfy the definition of the standard back-and-forth relations (Definition \ref{def:bf}), that is, if:
\begin{enumerate}
	\item If $\alpha$ is the smallest element of $L$, $\bar{a} \precsim_\alpha \bar{b}$ if for each quantifier-free formula $\psi(\bar{x})$ with G\"odel number less than the length of $\bar{a}$, if $\mc{A} \models \psi(\bar{a})$ then $\mc{A} \models \psi(\bar{b})$.
	\item If $\alpha$ is not the smallest element of $L$, $\bar{a} \precsim_\alpha \bar{b}$ if for each $\beta < \alpha$, for all $\bar{d}$ there is $\bar{c}$ such that $\bar{b}, \bar{d} \precsim_\beta \bar{a}, \bar{c}$.
\end{enumerate}
\end{definition}
While the standard back-and-forth relations are uniquely defined, this is not the case for non-standard back-and-forth relations. However, they are uniquely determined on the well-founded part of $L$.

\begin{remark}
Let $(L,<)$ be a linear order and $(\precsim_\alpha)_{\alpha \in L}$ a sequence of non-standard back-and-forth relations on $\mc{A}$. The relations $\precsim_\alpha$ for $\alpha \in \wf(L)$ are the same as the standard back-and-forth relations $\leq_\alpha$ on $\mc{A}$.
\end{remark}

For non-standard $\alpha \in L$, that is, $\alpha \in L \setminus \wf(L)$, the back-and-forth relations hold only between tuples in the same automorphism orbit.

\begin{lemma}\label{lem:non-standard-bf}
Let $(L,<)$ be a linear order and $(\precsim_\alpha)_{\alpha \in L}$ a sequence of non-standard back-and-forth relations on $\mc{A}$. For $\alpha \in L \setminus \wf(L)$, if $\bar{a} \precsim_\alpha \bar{b}$, then there is an isomorphism of $\mc{A}$ taking $\bar{a}$ to $\bar{b}$.
\end{lemma}
\begin{proof}
It is easy to see that
\[ \{ \bar{a} \mapsto \bar{b} : \bar{a} \leq_\beta \bar{b} \text{ for some } \beta \in L \setminus \wf(L)\} \]
is a set of finite maps with the back-and-forth property. If $\bar{a} \leq_\beta \bar{b}$ for some $\beta \in L \setminus \wf(L)$, then $\bar{a}$ and $\bar{b}$ satisfy the same atomic sentences. Thus any such map extends to an automorphism.
\end{proof}

\subsection{Admissible ordinals and Harrison linear orders}\label{sub:harrison}

Given $X \subseteq \omega$, $\omega_1^X$ is the least non-$X$-computable ordinal. By a theorem of Sacks \cite{Sacks76}, the countable admissible ordinals $\alpha > \omega$ are all of the form $\omega_1^{X}$ for some set $X$. For our purposes, we may take this as the definition of an admissible ordinal.

Harrison \cite{Harrison68} showed that for each $X \subseteq \omega$, there is an $X$-computable ordering which is not well-ordered, but which has no $X$-hyperarithmetic descending sequence. Moreover, any such ordering is of order type $\omega_1^X \cdot (1 + \mathbb{Q}) + \beta$ for some $X$-computable ordinal $\beta$. We call $\omega_1^X \cdot (1 + \mathbb{Q})$ the Harrison linear order relative to $X$. Note that the property of being the Harrison linear order relative to $X$ is $\Sigma^1_1(X)$: a linear order is the Harrison linear order relative to $X$ if:
\begin{enumerate}
	\item it is $X$-computable,
	\item for every $X$-computable ordinal $\alpha$ and element $x$, there is $y$ such that the interval $[x,y)$ has order type $\alpha$,
	\item it has a descending sequence, and
	\item for every $X$-computable ordinal $\alpha$ and index $e$ there is a jump hierarchy on $\alpha$ which witnesses that $\varphi_e^{0^{(\alpha)}}$ is not a descending sequence.
\end{enumerate}

Later we will use the fact that the set of admissible ordinals contains a club.

\begin{definition}
A set $U \subseteq \omega_1$ is closed unbounded (club) if it is unbounded below $\omega_1$ and is closed in the order topology, i.e., if $\sup(U \cap \alpha) = \alpha \neq 0$, then $\alpha \in U$.
\end{definition}

\begin{definition}
A set $U \subseteq \omega_1$ is stationary if it intersects every club set.
\end{definition}

\begin{remark}\label{rem:club}
Given a set $Y \subseteq \omega$, the set of $\alpha < \omega_1$ such that $L_{\alpha}[Y]$ is an elementary substructure of $L_{\omega_1}[Y]$ is a club. Hence the set $\{ \omega_1^X \colon X \geq_T Y\}$ contains a club. (Recall also that every club is a stationary set.)
\end{remark}

\section{The Main Construction}\label{sec:gen-const}
In this section we will do the main work of this paper by giving the general construction used in the applications. Given an $\mc{L}_{\omega_1 \omega}$-pseudo-elementary class $\mc{S}$ of linear orders, we will build a theory $T$ whose models have Scott ranks in correspondence with the linear orders in $\mc{S}$.

\begin{theorem}\label{thm:construction}
Let $\mc{S}$ be a $PC_{\mc{L}_{\omega_1 \omega}}$-class of linear orders. Then there a an $\mc{L}_{\omega_1 \omega}$-sentence $T$ such that
\[ \scs(T) = \{ \wfc(L) : L \in \mc{S} \}.\]
Moreover, suppose that $\mc{S}$ is the class of reducts of a sentence $S$. Then:
\begin{enumerate}
	\item We can choose $T$ to be $\Pi^{\infi}_2$ (or $\Pi^{\comp}_2$ if $S$ is computable).
	\item If $L$ is a computable model of $S$ with a computable successor relation, then there is a computable model $\mc{M}$ of $T$ with $SR(\mc{M}) = \wfc(L)$.
\end{enumerate}
\end{theorem}

With a little more work, we can replace the well-founded collapse with the well-founded part:

\begin{theorem}\label{thm:construction2}
In Theorem \ref{thm:construction}, we can also get
\[ \scs(T) = \{ \wf(L) : L \in \mc{S} \}.\]
\end{theorem}

\subsection{Overview of the construction}

Our structures will have two sorts, the \textit{order sort} and the \textit{main sort}. We will also treat elements of $\omega$ as if they are in the structure (e.g., we will talk about functions with codomain $\omega$). We can identify $\mc{S}$ with an $\mc{L}_{\omega_1 \omega}$ sentence $S$ in the language with a symbol $\leq$ for the ordering and possibly further symbols; $\mc{S}$ is the class of reducts of models of $S$ to the language with just the symbol $\leq$. Let $S^+$ be $S$ together with:
\begin{enumerate}[label=(O\arabic*)]
	\item\label{O1} There are constants $(e_i)_{i \in \omega}$ such that each element is equal to exactly one constant.
	\item\label{O2} There is a partial successor function $\alpha \mapsto \alpha+1$, and each non-maximal element has a successor.
	\item\label{O3} There is a sequence $(R_n)_{n \in \omega}$ of subsets satisfying:
	\begin{enumerate}[label=(R\arabic*)]
		\item\label{R1} $R_1$ is not strictly bounded (i.e., there is no $\alpha$ which is strictly greater than each element of $R_1$),
		\item\label{R2} $R_n \subseteq R_{n+1}$,
		\item\label{R3} $\bigcup_n R_n$ is the whole universe of the order sort.
		\item\label{R4} If $\alpha \in R_n$, then $\alpha = \sup(\beta + 1 : \beta \in R_{n+1} \text{ and } \beta < \alpha)$,
		\item\label{R5} For each $n$ and $\beta$, there is a least element $\gamma$ of $R_n$ with $\gamma \geq \beta$.
	\end{enumerate}
\end{enumerate}
For Theorem \ref{thm:construction} (i.e.\ to get $\scs(T) = \{ \wfc(L) : L \models S \}$) we will add
\begin{enumerate}[label=(O4a)]
	\item\label{O4a} $R_n = L$ for all $n$.
\end{enumerate}
\ref{O3} is a consequence of \ref{O4a}; moreover, \ref{O4a} will make the $R_n$ trivial (see \ref{Q6} below). \ref{O4a} will only be used for the final computation of the Scott ranks of the models of $T$, whereas \ref{O3} will be used in the construction itself. For Theorem \ref{thm:construction2}, we will use a different axiom \ref{O4b} instead of \ref{O4a}; \ref{O3} will also be a consequence of \ref{O4b}. The general construction will be the same, but \ref{O4b} will give us a different computation of the Scott rank of the resulting models. Thus \ref{O3} is exactly that common part of \ref{O4a} and \ref{O4b} which is required for the construction, and the particulars of \ref{O4a} and \ref{O4b} are what give the Scott ranks. While reading through the construction for the first time, it might be helpful to assume that \ref{O4a} is in effect. Each order type in $\mc{S}$ is represented as a model of $S^+$.

The order sort will be a model of $S^+$. Our next step will be to define, for each model $L$ of $S^+$, an $\mc{L}_{\omega_1 \omega}$-sentence $T(L)$. The sentence $T$ will say that the order sort is a model $L$ of $S^+$ and the main sort is a model of $T(L)$. In defining $T(L)$, we will use quantifiers over $L$, and $T(L)$ will be uniform in $L$.

For now, fix a particular model $L$ of $S^+$. As a model of $S$, $L$ will be a linear ordering, which we view as a non-standard ordinal. The Scott rank of $\mc{M} \models T(L)$ will be determined by $L$; in particular, if $L$ is actually an ordinal, then the Scott rank of $\mc{M}$ will be its order type. If $(L,\mc{M})$ is a model of $T$, then since by \ref{O1} each element of $L$ is named by a constant, the Scott rank of $L$ will be as low as possible, and so the Scott rank will be carried by $\mc{M}$. We will have
\[\scs(T) = \{ SR(\mc{M}) \colon \mc{M} \models T(L) \text{ for some $L \models S^+$} \}. \]
If $L$ has a least element and at least two elements, then for $\mc{M} \models T(L)$, $SR(\mc{M})$ will be $\wfc(L)$ (or $\wf(L)$ in the case of Theorem \ref{thm:construction2}). We can then modify $T$ slightly using Proposition \ref{prop:change-spec} to get the theorem; we first modify $S$ so that every $L \models S$ has a least element and at least two elements, and then we use Proposition \ref{prop:change-spec} to add $0$ or $1$, if desired, to the Scott spectrum. Since there are structures of Scott rank $0$ and $1$ which have Scott sentences which are $\Pi^{\comp}_1$ and $\Pi^{\comp}_2$ respectively, Proposition \ref{prop:change-spec} gives the correct quantifier complexity.

$T(L)$ will be constructed as follows. First, we will let $\mathbb{K}$ be the class of finite structures satisfying the properties \ref{P1}-\ref{P6} and \ref{Q0}-\ref{Q6} below. We will show that $\mathbb{K}$ has a Fra\"{\i}ss\'e limit. This is an ultrahomogeneous structure, and hence has very low Scott rank. We will add to the Fra\"{\i}ss\'e limit unary relations $A_i$ indexed by $i \in \omega$. $T(L)$ will be a sentence of $\mc{L}_{\omega_1 \omega}$ defining the Fra\"{\i}ss\'e limit of $\mathbb{K}$ together with relations $A_i$ satisfying properties \ref{A2} and \ref{A1}.

To see (1) of Theorem \ref{thm:construction}, we can take the Morleyization of $S^+$. This will be a $\Pi^{\infi}_2$ sentence which defines the same class of linear orders. The construction of $T(L)$ relative to $L$ is $\Pi^{\infi}_2$, so if we define $T$ in the same way as above but replacing $S^+$ by its Morleyization $S^+_{M}$, $T$ will now be $\Pi^{\infi}_2$. Since in each model $(L,\mc{M})$ of $T$, each element of $L$ is named by a constant, we still have
\[\scs(T) = \{ SR(\mc{M}) \colon \mc{M} \models T(L) \text{ for some $L \models S^+_M$} \}. \]
If $S$ is actually a computable formula, then its Morleyization is computable, and this $T$ will be computable.

To see (2), we observe that if $L$ is a computable model of $S$ with a computable successor relation, then it has a computable expansion to a model of \ref{O4a} (and hence of \ref{O3}). Then Lemma \ref{lem:fraisse-comp} below will show that there is a computable model of $T$ with order sort $L$.

\subsection{\texorpdfstring{The definition of $T(L)$}{The definition of T(O)}}

Fix $L \models S^+$. We begin by constructing the age of our Fra\"{\i}ss\'e limit. Let $\mathbb{K}$ be be the class of finite structures $\mc{M}$ satisfying \ref{P1}-\ref{P6} and \ref{Q0}-\ref{Q6} below. Structures in $\mathbb{K}$ should be viewed as trees.

\begin{enumerate}[label=(P\arabic*)]
	\item\label{P1} $\preceq$ a partial tree-ordering, that is, the set of predecessors of any element is linearly ordered.
	\item\label{P2} $\la \ra$ is the unique $\preceq$-smallest element.
	\item\label{P3} Each element other than $\la \ra$ has a unique predecessor, and $P$ is a unary function $\mc{M} \to \mc{M}$ picking out that predecessor.
	\item\label{P4} Each element has finite length, i.e., there is a finite chain of successors starting at $\langle \rangle$ and ending at that element.
	\item\label{P5} $\varrho \colon \mc{M} \setminus \{ \la \ra \} \to L$ and $\varepsilon \colon \mc{M} \setminus \{ \la \ra \} \to \omega$ are unary functions.
	\item\label{P6} If $x \prec y$, then $\varrho(x) > \varrho(y)$.
\end{enumerate}

The properties \ref{P1}-\ref{P6} that we have introduced so far already define the age of a Fra\"{\i}ss\'e limit in the restricted language $\{\la \ra, \preceq,P,\varrho,\varepsilon\}$. In reading the properties \ref{Q0}-\ref{Q6} below, it will helpful to have this model in mind.

\begin{lemma}\label{lem:first-fraisse}
The class of finite structures in the language $\{\la \ra, \preceq, P, \varrho,\varepsilon\}$ satisfying \ref{P1}-\ref{P6} has the hereditary property (HP), the amalgamation property (AP), and the joint embedding property (JEP). The Fra\"{\i}ss\'e limit is (isomorphic to) the following structure $\mc{M}$.

Fix an infinite set $D$. The domain of $\mc{M}$ is the set of all finite sequences
\[ \sigma = \la (\alpha_0,c_0,d_0),\ldots,(\alpha_n,c_n,d_n) \ra \]
with $\alpha_i \in L$, $c_i \in \omega$, and $d_i \in D$, such that $\alpha_0 > \alpha_1 > \cdots > \alpha_n$. We interpret the relations in the natural way: $\preceq$ is the standard ordering of extensions of sequences, $P$ is the standard predecessor function, $\varepsilon(\sigma) = c_n$, and $\varrho(\sigma)=\alpha_n$.
\end{lemma}
\begin{proof}
First, it is easy to see that the age of $\mc{M}$ is the set of finitely generated structures satisfying \ref{P1}-\ref{P6}. Then we just have to note that $\mc{M}$ is ultrahomogeneous to see that it is the Fra\"{\i}ss\'e limit of these structures.
\end{proof}

\noindent Given an element
\[ \sigma = \la (\alpha_0,c_0,d_0),\ldots,(\alpha_n,c_n,d_n) \ra \]
of this structure, write $\bar{\varrho}(\sigma)$ for $\la \alpha_0,\ldots,\alpha_n \ra$ and $\bar{\varepsilon}(x)$ for $\la c_0,\ldots,c_n \ra$. Write $|\sigma| = n+1$ for the length of $\sigma$.

We will now add an additional function $E$ whose properties are axiomatized by \ref{Q0}-\ref{Q6}. $E$ is a function from $\mc{M} \setminus \{ \la \ra\} \times \mc{M} \setminus \{ \la \ra \}$ to $\{-\infty\} \cup L \times \omega$. $E$ is defined only on those pairs $(x, y)$ with $|x| = |y|$, $\bar{\varrho}(x) = \bar{\varrho}(y)$, and $\bar{\varepsilon}(x) = \bar{\varepsilon}(y)$. Note that the domain of $E$ is an equivalence relation, for which we write $\between$. For convenience, when we talk about $E(x,y)$ for some $x$ and $y$ we will often implicitly assume that $x \between y$. We view the range of $E$ as a totally ordered set via the lexicographic ordering on $L \times \omega$, with $-\infty$ smaller than every element of $L \times \omega$. Given $x,y \in \mc{M}$ with $E(x,y) > -\infty$, let $E_{L}(x,y)$ be the first coordinate of $E(x,y)$, i.e., the coordinate in $L$, and let $E_{\omega}(x,y)$ be the second coordinate. If $E(x,y) = -\infty$, then we let $E_{L}(x,y) = -\infty$.

One can view $E$ as a nested sequence $(\sim_{\alpha,n})_{\alpha \in L, n \in \omega}$ of relations on $M$, defined by $x \sim_{\alpha,n} y$ if $E(x,y) \geq \min((\varrho(x),0),(\alpha,n))$. If $E(x,y) = -\infty$, then $x$ and $y$ are not at all related. It will follow from \ref{Q0}, \ref{Q1}, and \ref{Q2} that these are equivalence relations. These equivalence relations are nested and continuous (i.e., $\sim_{\alpha,0} = \bigcap_{\beta < \alpha,n\in\omega} \sim_{\beta,n}$). The most important relations are the relations $\sim_{\alpha,0}$ which we will denote by $\sim_\alpha$. The relations $\sim_\alpha$ will be non-standard back-and-forth relations (see Lemma \ref{lem:are-ns-bf}). The definition of the back-and-forth relations is not $\Pi^{\infi}_2$, so we cannot just ask that $\sim_\alpha$ satisfy the definition of the back-and-forth relations. This is where we use $\varepsilon$ and the $\omega$ in $L \times \omega$; their role is to convert an existential quantifier into a universal quantifier by acting as a sort of Skolem function.

If $x \in M$ is not a dead end, the children of $x$ are divided into infinitely many subsets indexed by $\omega$ via the function $\varepsilon$. If $E_{L}(x,y) > \alpha$, then for every child $x'$ of $x$, there will be a child $y'$ of $y$ with $E_{L}(x',y') \geq \alpha$; this is in keeping with the idea of making the equivalence relations $\sim$ agree with the back-and-forth relations. If $E_{L}(x,y) = \alpha$, then this will not be true for all $x'$. However, it will be true for exactly those $x'$ with $\varepsilon(x') < E_{\omega}(x,y)$. Rather than saying that there is a child $x'$ of $x$ such that no child $y'$ of $y$ has $E_{L}(x',y') = \alpha$, we can say that for all children $x'$ of $x$ with $\varepsilon(x) \geq E_{\omega}$, there is no child $y'$ of $y$ with $E_{L}(x',y') = \alpha$. This is of lower quantifier complexity. (Note that we cannot say that for \textit{all} $x'$ and $y'$ children of $x$ and $y$, $E(x',y') < \alpha$. This is for the same reason as the following fact: if $x$ and $y$ are such that for all $\bar{x}'$ and $\bar{y}'$, $x \bar{x}' \nequiv_{\alpha} y \bar{y}'$, then $x \nequiv_\alpha y$.)

For all $x$, $y$, and $z$ with $x \between y \between z$:
\begin{enumerate}[label=(Q\arabic*)]
	\item\label{Q0} $E(x,x) = (\varrho(x),0)$,
	\item\label{Q1} $E(x,y) = E(y,x)$,
	\item\label{Q2} $E(x,z) \geq \min(E(x,y),E(y,z))$,
	\item\label{Q3} $E(x,y) \leq (\varrho(x),0) = (\varrho(y),0)$.
	\item\label{Q5} If $x'$ and $y'$ are successors of $x$ and $y$ with $x' \between y'$, $E_{L}(x',y') \leq E_{L}(x,y)$.
	\item\label{Q4} If $E(x,y) > -\infty$, then for every $x'$ a successor of $x$ with $\varepsilon(x') \geq E_\omega(x,y)$, there are no successors $y'$ of $y$ with $E_{L}(x',y') = E_{L}(x,y)$.
	\item\label{Q6} If $|x| = |y| = n$, then $E_{L}(x,y) \in R_n \cup \{-\infty\} \cup \{ \varrho(x) \}$.
\end{enumerate}

While $\la \ra$ was not in the domain of $E$, we will consider $E(\la \ra, \la \ra)$ to be $(L,0)$, i.e., to be greater than each element of $L$. 

\ref{Q0}, \ref{Q1}, and \ref{Q2} are just saying that the relations $\sim_{\alpha,n}$ defined above are reflexive, symmetric, and transitive respectively (and hence equivalence relations). \ref{Q4} is the axiom which is doing most of the work.

The intuition behind \ref{Q6} will be explained in Subsection \ref{subsec:omegaCK}. For now, the reader can simply imagine that $R_n = L$ for each $n$ (as it will be for Theorem \ref{thm:construction}), so that \ref{Q6} is a vacuous condition.

\begin{lemma}\label{lem:second-fraisse}
The class $\mathbb{K}$ of finite structures satisfying \ref{P1}-\ref{P6} and \ref{Q0}-\ref{Q6} relative to the fixed structure $L$ has the AP, JEP, and HP.
\end{lemma}
\begin{proof}
It is easy to see that $\mathbb{K}$ has the hereditary property. Note that every finite structure in $\mathbb{K}$ contains, via an embedding, the structure with one element $\langle \rangle$. So the joint embedding property will follow from the amalgamation property.

For the amalgamation property, let $\mc{A}$ be a structure in $\mathbb{K}$ which embeds into $\mc{B}$ and $\mc{C}$. Identify $\mc{A}$ with its images in $\mc{B}$ and $\mc{C}$, and assume that the only elements common to both $\mc{B}$ and $\mc{C}$ are the elements of $\mc{A}$. By amalgamating $\mc{C}$ one element at a time, we may assume that $\mc{C}$ contains only a single element $c$ not in $\mc{A}$. The element $c$ is the child of some element of $\mc{A}$, and has no children in $\mc{C}$. We will define a structure $\mc{D}$ whose domain is $\mc{B} \cup \{c\}$ and then show that $\mc{D}$ is in $\mathbb{K}$.

First, we can take the amalgamation of the structures in the language $\{\langle \rangle,\preceq,P,\varrho,\varepsilon\}$ as in Lemma \ref{lem:first-fraisse}; we just add $c$ to $\mc{B}$, setting $P(c)$, $\varepsilon(c)$, and $\varrho(c)$ to be the same as in $\mc{C}$. Set $E(c,c) = (\varrho(c),0)$. We need to define $E(b,c)$ when $b$ is an element from $\mc{B}$ with $b \between c$. Define $E(b,c) = E(c,b)$ to be the maximum of $\min(E(b,a),E(a,c))$ over all $a \in \mc{A}$ with $a \between c$. If there are no such $a \in \mc{A}$, set $E(b,c) = -\infty$. By \ref{Q2} this is well-defined, that is, for $a \in \mc{A}$, $E(a,c)$ is the maximum of $\min(E(a,a'),E(a',c))$ over all $a' \in \mc{A}$.

We now have to check \ref{Q0}-\ref{Q6}. \ref{Q0} and \ref{Q1} are obvious from the definition of the extension of $E$.

For \ref{Q2}, we have two new cases to check. For the first case, fix $b, b' \in \mc{B}$ with $b \between b' \between c$; we will show that $E(b,c) \geq \min(E(b,b'),E(b',c))$. If there is no $a \in \mc{A}$ with $a \between c$, then $E(b,c) = E(b',c) = -\infty$, and so we have $E(b,c) \geq \min(E(b,b'),E(b',c))$. Otherwise, let $a \in \mc{A}$ be such that $E(b',c) = \min(E(b',a),E(a,c))$. By definition,
\begin{align*}
E(b,c) &\geq \min(E(b,a),E(a,c)) \\
&\geq \min(E(b,b'),E(b',a),E(a,c)) \\
&= \min(E(b,b'),E(b',c)).
\end{align*}

Now for the second case, again fix $b, b' \in \mc{B}$ with $b \between b' \between c$; now we will show that $E(b,b') \geq \min(E(b,c),E(c,b'))$. If there is no $a \in \mc{A}$ with $a \between c$, then $E(b,c) = E(b',c) = -\infty$, and so we have
\[ E(b,b') \geq -\infty = \min(E(b,c),E(c,b')).\]
Otherwise, let $a \in \mc{A}$ be such that $E(b,c) = \min(E(b,a),E(a,c))$, and let $a' \in \mc{A}$ be such that $E(b',c) = \min(E(b',a'),E(a',c))$. Then
\begin{align*}
E(b,b') &\geq \min(E(b,a),E(a,a'),E(a',b')) \\
& \geq \min(E(b,a),E(a,c),E(c,a'),E(a',b')) \\
&= \min(E(b,c),E(c,b')).
\end{align*}

For \ref{Q3}, suppose that $b \in \mc{B}$ has $b \between c$. Then either $E(b,c) = -\infty$, in which case there is nothing to check, or $E(b,c) = \min(E(b,a),E(a,c))$ for some $a \in \mc{A}$. In the second case, either $E(b,c) = E(b,a) \leq (\varrho(a),0) = (\varrho(c),0)$ or $E(b,c) = E(a,c) \leq (\varrho(c),0)$.

Now we check \ref{Q5}. Let $\hat{c} \in \mc{A}$ be the parent of $c$. Fix $b \in \mc{B}$ and let $\hat{b}$ be the parent of $b$. We must show that $E_{L}(b,c) \leq E_{L}(\hat{b},\hat{c})$. Choose $a \in \mc{A}$ such that $E(b,c) = \min(E(b,a),E(a,c))$. (If there is no such $a$, then we can immediately see that \ref{Q5} holds.) Let $\hat{a}$ be the parent of $a$. Then we have $E_{L}(\hat{b},\hat{a}) \geq E_{L}(b,a)$ and $E_{L}(\hat{a},\hat{c}) \geq E_{L}(a,c)$ so that
\begin{align*}
E_{L}(\hat{b},\hat{c}) &\geq \min(E_{L}(\hat{b},\hat{a}),E_{L}(\hat{a},\hat{c})) \\
& \geq \min(E(b,a),E(a,c)) \\
& = E_{L}(b,c).
\end{align*}

Next we check \ref{Q4}. Since $c$ has no children in $\mc{C}$, the only new case to check is as follows. Let $\hat{c} \in \mc{A}$ be the parent of $c$, and let $\hat{b} \in \mc{B}$ be such that $\hat{c} \between \hat{b}$. Suppose that $n = \varepsilon(c) \geq E_\omega(\hat{c},\hat{b})$ and let $\alpha = E_{L}(\hat{c},\hat{b}) > -\infty$. Suppose to the contrary that there is $b$ a child of $\hat{b}$ with $E_{L}(b,c) = \alpha$. Then, by definition there is $a \in \mc{A}$ such that $E(b,c) = \min(E(b,a),E(a,c))$. Let $\hat{a}$ be the parent of $a$. Since $E_{L}(b,a) \geq \alpha$ and $E_{L}(a,c) \geq \alpha$, $E(\hat{b},\hat{a}) > (\alpha,n)$ and $E(\hat{a},\hat{c}) > (\alpha,n)$. Hence $E(\hat{b},\hat{c}) > (\alpha,n)$. This is a contradiction.

Finally, for \ref{Q6}, if $E(b,c) = -\infty$ we are done. So we may suppose that $E(b,c) = \min(E(b,a),E(a,c))$ for some $a \in \mc{A}$. Then since $E_{L}(b,c)$ and $E_{L}(a,c)$ are both in $R_n \cup \{-\infty\} \cup \{\varrho(c)\}$, the same is true of $E_{L}(b,c)$.
\end{proof}

\begin{lemma}\label{lem:third-fraisse}
The reduct of the Fra\"{\i}ss\'e limit of $\mathbb{K}$ to the language $\{\la \ra,{\preceq,} P,\varrho,\varepsilon\}$ is the structure from Lemma \ref{lem:first-fraisse}.
\end{lemma}
\begin{proof}
We just need to show that if $\mc{M}$ is a structure in $\mathbb{K}$, and $\mc{N}$ is a structure in the language $\mc{L}_{-} = \{\la \ra,\preceq,P,\varrho,\varepsilon\}$ satisfying \ref{P1}-\ref{P6} and with $\mc{M} \subseteq_{\mc{L}_{-}} \mc{N}$, then we can expand $\mc{N}$ to a structure $\mc{N}'$ in the language $\mc{L} = \mc{L}_{-} \cup \{E\}$ with $\mc{M} \subseteq_{\mc{L}} \mc{N}'$. We can do this simply by setting $E(x,x) = (\varrho(x),0)$ for $x \in \mc{N} \setminus \mc{M}$, and $E(x,y) = E(y,x) = -\infty$ for all $x \in \mc{N} \setminus \mc{M}$ and $y \in \mc{N}$. \ref{Q0}-\ref{Q6} are easy to check.
\end{proof}

For a fixed $L$, let $T(L)$ be the $\mc{L}_{\omega_1 \omega}$-sentence describing the Fra\"{\i}ss\'e limit of $\mathbb{K}$, and to which we add unary relations $(A_i)_{i \in \omega}$ satisfying \ref{A2} and \ref{A1} below. The relations $A_i$ will name the equivalence classes $\sim_0$, so that while the Fra\"{\i}ss\'e limit is ultra-homogeneous, the models of $T(L)$ will not be. The Fra\"{\i}ss\'e limit is axiomatizable by a $\Pi^{\infi}_2$ formula. If $L$ is computable with a computable successor relation, then $\mathbb{K}$ is a computable age, and hence the Fra\"{\i}ss\'e limit is axiomatizable by a $\Pi^{\comp}_2$ formula. Since \ref{P1}-\ref{P6} and \ref{Q0}-\ref{Q6} are all $\Pi^{\infi}_2$, they hold in the Fra\"{\i}ss\'e limit. Since \ref{A2} and \ref{A1} are also $\Pi^{\infi}_2$ formulas, $T(L)$ is $\Pi^{\infi}_2$ axiomatizable.

\begin{enumerate}[label=(A\arabic*)]
\item\label{A2} For each $x$, $A_i(x)$ for exactly one $i$.

\item\label{A1} For all $x$ and $x'$, $E(x,x') > -\infty$ if and only if for all $i$, $A_{i}(x)\Leftrightarrow A_{i}(x')$.
\end{enumerate}

\begin{lemma}\label{lem:fraisse-comp}
If $L$ is computable with a computable successor relation, then $T(L)$ has a computable model.
\end{lemma}
\begin{proof}
By Theorem 3.9 of \cite{CsimaHarizanovMillerMontalban11}, there is a computable Fra\"{\i}ss\'e limit of $\mathbb{K}$. Then we can add the relations $A_i$ in a computable way.
\end{proof}

\subsection{Computation of the Scott rank}

Fix $L$ a model of $S^+$. Let $\mc{M}$ be a countable model of $T(L)$. The remainder of the proof is a computation of $SR(\mc{M})$. As remarked earlier, for this section we will assume that $L$ has a least element $0$ and has at least two elements. We will show that $SR(\mc{M}) = \wfc(L)$ for such an $\mc{M}$. Let $\wf(L)$ be the well-ordered part of $L$. Recall that we identify elements of $\wf(L)$ with ordinals in the natural way.

\begin{lemma}\label{lem:extension}
Fix $\beta \in \wf(O)$. Suppose that $m \in \omega$ and $u_1,\ldots,u_t$, $u_1',\ldots,u_t'$, and $v$ are tuples from $\mc{M}$ such that $\varepsilon(x) < m$ where $x$ ranges among all of these elements and their predecessors, and such that:
\begin{enumerate}
	\item[(i)] $u_1,\ldots,u_t \equiv_{\at} u_1',\ldots,u_t'$,
	\item[(ii)] for each $i$, $E(u_i,u_i') \geq \min((\varrho(u_i),0),(\beta,m))$.
\end{enumerate}
Suppose moreover that $u_1,\ldots,u_t$ and $u_1',\ldots,u_t'$ are closed under the predecessor relation $P$.
Then there is $v'$ such that $u_1,\ldots,u_t,v$ and $u_1',\ldots,u_t',v'$ satisfy (i) and (ii).
\end{lemma}
\begin{proof}
We may assume that $v$ is not one of $u_1,\ldots,u_t$, as if $v = u_i$ then we could take $v' = u_i'$. Thus for no $u_i$ is $u_i \succeq v$. By repeated applications of the claim, we may also assume that $P(v)$ is among the $u_i$.

Let $u$ be the predecessor of $v$, and let $y_1,\ldots,y_k$ be those $u_i$ with $v \between u_i$. Let $x_1,\ldots,x_k$ be the predecessors of the $y_i$. Let $u'$, $y_1',\ldots,y_k'$, and $x_1',\ldots,x_k'$ be the corresponding $u_i'$. We will define a finite structure with domain consisting of $u_1,\ldots,u_t$, $u_1',\ldots,u_t'$, $v$, and a new element $v'$. We will show that this structure is in $\mathbb{K}$, and hence we may take $v'$ to be in $\mc{M}$.

Begin by defining $|v'| = |v|$, $\varepsilon(v') = \varepsilon(v)$, $\varrho(v') = \varrho(v)$, and $u_i' \prec v'$ if and only if $u_i \prec v$. Define $E$ as follows:
\begin{enumerate}
	\item $E(v',v') = (\varrho(v'),0)$,
	\item $E(v',y_i') = E(v,y_i)$.
	\item $E(v',y_i)$ is the maximum of $\min(E(v',y_j'),E(y_j',y_i))$ over all $j$,
	\item If $\beta \in R_{|v|}$, then $E(v',v)$ is the maximum of:
	\begin{enumerate}
		\item[(a)] $\min(E(v',y_i'),E(y_i',v))$ over all $i$ and
		\item[(b)] $\min((\varrho(v),0),(\beta,m))$.
	\end{enumerate}
	Otherwise, if $\beta \notin R_{|v|}$, then let $\gamma$ be the least element of $R_{|v|}$ with $\gamma > \beta$. (\ref{R5} guarantees that such a $\gamma$ exists.) Then $E(v',v)$ is the maximum of:
	\begin{enumerate}
		\item[(a)] $\min(E(v',y_i'),E(y_i',v))$ over all $i$ and
		\item[(c)] $\min((\varrho(v),0),(\gamma,0))$.
	\end{enumerate}
\end{enumerate}
Let $E(\cdot,v') = E(v',\cdot)$ in each of the cases above. Note that by (2), (3) is equivalent to defining $E(v',y_i)$ to be the maximum of $\min(E(v,y_j),E(y_j',y_i))$ over all $j$, and similarly with (4); so these are all definitions in terms of quantities we are given.

We need to check that this is defines a finite structure in $\mathbb{K}$. It is easy to see that \ref{P1}-\ref{P6} hold. \ref{Q0} and \ref{Q1} are trivial, as we set $E(v',v') = (\varrho(v'),0)$ and $E(\cdot,v') = E(v',\cdot)$ above. \ref{Q3} is easy to see from the definition of $E(v',\cdot)$.  We now check \ref{Q2}, \ref{Q5}, \ref{Q4}, and \ref{Q6}.

For \ref{Q2}, we must show that if $a \between b \between c$, then we have $E(a,b) \geq \min(E(a,c),E(c,b))$. We have a number of different cases depending on which values $a$, $b$, and $c$ take. If none of $a$, $b$, or $c$ are $v'$, then it is trivial; also, if there is any duplication, then it is trivial. Unfortunately there are a large number of possible combinations remaining. The reader might find it helpful to draw a picture for each case, using the intuition of \ref{Q2} as corresponding to the transitivity of an equivalence relation. We will frequently use the fact that \ref{Q2} holds in $\mc{M}$.

\begin{description}
	
	\item[$a = v'$, $b = v$, $c = y_i$] Let $j$ be such that $E(v',y_i) = \min(E(v',y_j'),E(y_j',y_i),)$. Then
\begin{align*}
E(v',v) &\geq \min(E(v',y_j'),E(y_j',v)) \\
& \geq \min(E(v',y_j'),E(y_j',y_i),E(y_i,v)) \\
&= \min(E(v',y_i),E(y_i,v)).
\end{align*}

	\item[$a = v'$, $b = v$, $c = y_i'$] $E(v',v) \geq \min(E(v',y_i'),E(y_i',v))$ by definition.
	
	\item[$a=v'$, $b = y_i$, $c = v$] We have three cases corresponding to (a), (b), and (c) in the definition of $E(v',v)$.

	\begin{enumerate}
		
				\item[(a)] Suppose that $E(v',v) = \min(E(v',y_j'),E(y_j',v))$ for some $j$. Then 
\begin{align*}
E(v',y_i) &\geq \min(E(v',y_j'),E(y_j',y_i)) \\
&\geq \min(E(v',y_j'),E(y_j',v),E(v,y_i)) \\ %\qquad \text{ by \ref{Q2}} \\
&= \min(E(v',v),E(v,y_i)).
\end{align*}
		
		\item[(b)] Suppose that $E(v',v) = \min((\varrho(v),0),(\beta,m))$. We know that $E(y_i',y_i) \geq \min((\varrho(v),0),(\beta,m))$, and so $E(y_i',y_i) \geq E(v',v)$. Then
		\[ E(v',y_i) \geq \min(E(v',y_i'),E(y_i',y_i)) \geq \min(E(v',v),E(v,y_i)).\]
	
		\item[(c)] Suppose that $E(v',v) = \min((\varrho(v),0),(\gamma,0))$. We know that $E(y_i',y_i) \geq \min((\varrho(v),0),(\beta,m))$, and since $E_{L}(y_i',y_i) \in R_{|v|} \cup \{- \infty\} \cup \{\varrho(v)\}$, $E(y_i',y_i) \geq \min((\varrho(v),0),(\gamma,0))$. So $E(y_i',y_i) \geq E(v',v)$. Then
		\[ E(v',y_i) \geq \min(E(v',y_i'),E(y_i',y_i)) \geq \min(E(v',v),E(v,y_i)).\]

	\end{enumerate}
	
	\item[$a = v'$, $b = y_i$, $c = y_j$] Let $k$ be such that $E(v',y_j) = \min(E(v',y_k'),E(y_k',y_j))$. 
By definition, we have
\begin{align*}
E(v',y_i) &\geq \min(E(v',y_k'),E(y_k',y_i)) \\
&\geq \min(E(v',y_k'),E(y_k',y_j),E(y_j,y_i)) \\ %\qquad \text{ by \ref{Q2}}\\
&= \min(E(v',y_j),E(y_j,y_i)).
\end{align*}

	\item[$a=v'$, $b=y_i$, $c = y_j'$] By definition, $E(v',y_i) \geq \min(E(v',y_j'),E(y_j',y_i))$.
	
	\item[$a = v'$, $b = y_i'$, $c = v$] We have three cases corresponding to (a), (b), and (c) in the definition of $E(v',v)$.

	\begin{enumerate}
		\item[(b)] Suppose that $E(v',v) = \min(E(v',y_j'),E(y_j',v))$ for some $j$. We have
\begin{align*}
E(v',y_i') & = E(v,y_i) \\
&\geq \min(E(v,y_j),E(y_j,y_i)) \\%&&\text{ by \ref{Q2}}\\
&= \min(E(v',y_j'),E(y_j',y_i')) \\
&\geq \min(E(v',y_j'),E(y_j',v),E(v,y_i')) \\%&&\text{ by \ref{Q2}}\\
&= \min(E(v',v),E(v,y_i')).
\end{align*}
		
		\item[(b)] Suppose that $E(v',v) = \min((\varrho(v),0),(\beta,m))$. Then
\begin{align*}
E(v',y_i') &= E(v,y_i) \\
&\geq \min(E(v,y_i'),E(y_i',y_i)) \\
&\geq \min(E(v,y_i'),(\varrho(y_i),0),(\beta,m)) \\
&= \min(E(v,y_i'),E(v',v)).
\end{align*}

		\item[(c)] Suppose that $E(v',v) = \min((\varrho(v),0),(\gamma,0))$. Then, as before, $E(y_i',y_i)) \geq \min((\varrho(y_i),0),(\gamma,0))$. So
\begin{align*}
E(v',y_i') &= E(v,y_i) \\
&\geq \min(E(v,y_i'),E(y_i',y_i)) \\
&\geq \min(E(v,y_i'),(\varrho(y_i),0),(\gamma,0)) \\
&= \min(E(v,y_i'),E(v',v)).
\end{align*}
	
	\end{enumerate}
	
	\item[$a = v'$, $b = y_i'$, $c = y_j$] Let $k$ be such that $E(v',y_j) = \min(E(v',y_k'),E(y_k',y_j))$. We have
\begin{align*}
E(v',y_i') &= E(v,y_i) \\
&\geq \min(E(v,y_k),E(y_k,y_i)) \\%&&\text{ by \ref{Q2}}\\
&= \min(E(v',y_k'),E(y_k',y_i')) \\
&\geq \min(E(v',y_k'),E(y_k',y_j),E(y_j,y_i')) \\%&&\text{ by \ref{Q2}}\\
&= \min(E(v',y_j),E(y_j,y_i')).
\end{align*}
	
	\item[$a = v'$, $b = y_i'$, $c = y_j'$] We have
\begin{align*}
E(v',y_i') &= E(v,y_i) \\
&\geq \min(E(v,y_j),E(y_j,y_i)) \\%&&\text{ by \ref{Q2}}\\
&\geq \min(E(v',y_j'),E(y_j',y_i')).
\end{align*}

	\item[$a = y_i$, $b = y_j$, $c = v'$] Let $k$ and $\ell$ be such that $E(v',y_i) = \min(E(v',y_k'),E(y_k',y_i))$ and $E(v',y_j) = \min(E(v',y_\ell'),E(y_\ell',y_j))$. We have
\begin{align*}
E(y_i,y_j) &\geq \min(E(y_i,y_k'),E(y_k',y_\ell'),E(y_\ell',y_j)) \\%&&\text{ by \ref{Q2}}\\
&= \min(E(y_i,y_k'),E(y_k,y_\ell),E(y_\ell',y_j)) \\
&\geq \min(E(y_i,y_k'),E(y_k,v),E(v,y_\ell),E(y_\ell',y_j)) \\%&&\text{ by \ref{Q2}}\\
&= \min(E(y_i,y_k'),E(y_k',v'),E(v',y_\ell'),E(y_\ell',y_j)) \\
&= \min(E(y_i,v'),E(v',y_j)).
\end{align*}

	\item[$a = y_i$, $b = y_j'$, $c = v'$] Let $k$ be such that $E(v',y_i) = \min(E(v',y_k'),E(y_k',y_i))$. Then
\begin{align*}
E(y_i,y_j') &\geq \min(E(y_i,y_k'),E(y_k',y_j'))\\% &&\text{ by \ref{Q2}}\\
&= \min(E(y_i,y_k'),E(y_k,y_j)) \\
&\geq \min(E(y_i,y_k'),E(y_k,v),E(v,y_j)) \\%&&\text{ by \ref{Q2}}\\
&= \min(E(y_i,y_k'),E(y_k',v'),E(v',y_j')) \\
&= \min(E(y_i,v'),E(v',y_j')).
\end{align*}
	
	\item[$a = y_i'$, $b = y_j'$, $c = v'$] We have
\begin{align*}
E(y_i',y_j') &= E(y_i,y_j) \\
& \geq \min(E(y_i,v),E(v,y_j)) \\
&= \min(E(y_i',v'),E(v',y_j')).
\end{align*}

	\item[$a = v$, $b = y_i$, $c = v'$] Let $j$ be such that $E(v',y_i) = \min(E(v',y_j'),E(y_j',y_i))$. We have three cases corresponding to (a), (b), and (c) in the definition of $E(v',v)$.

		\begin{enumerate}
			\item[(a)] Suppose that $E(v',v) = \min(E(v',y_j'),E(y_j',v))$ for some $j$ and $E(v',y_i) = \min(E(v',y_k'),E(y_k',y_i))$ for some $k$. Then
\begin{align*}
E(v,y_i) &\geq \min(E(v,y_k'),E(y_k,y_j),E(y_j',y_i)) \\
&= \min(E(v,y_k'),E(y_k,v),E(v,y_j),E(y_j',y_i)) \\
&\geq \min(E(v,y_k'),E(y_k',v'),E(v',y_j'),E(y_j',y_i)) \\
&= \min(E(v,v'),E(v',y_i)).
\end{align*}
			
			\item[(b)] Suppose that $E(v',v) = \min((\varrho(v),0),(\beta,m))$. We have
\begin{align*}
E(v,y_i) &\geq \min(E(v,y_j),E(y_j,y_i)) \\%&&\text{ by \ref{Q2}}\\
&= \min(E(v',y_j'),E(y_j',y_i')) \\
&\geq \min(E(v',y_j'),E(y_j',y_i),E(y_i,y_i')) \\%&&\text{ by \ref{Q2}}\\
&= \min(E(v',y_i),E(y_i,y_i')) \\
&\geq \min(E(v',y_i),(\varrho(y_i),0),(\beta,m)) \\
&= \min(E(v',y_i),E(v',v)).
\end{align*}

			\item[(c)] Suppose that $E(v',v) = \min((\varrho(v),0),(\gamma,0))$. As before, $E(y_i,y_i') \geq \min((\varrho(v),0),(\gamma,0))$. We have
\begin{align*}
E(v,y_i) &\geq \min(E(v,y_j),E(y_j,y_i)) \\%&&\text{ by \ref{Q2}}\\
&= \min(E(v',y_j'),E(y_j',y_i')) \\
&\geq \min(E(v',y_j'),E(y_j',y_i),E(y_i,y_i')) \\%&&\text{ by \ref{Q2}}\\
&= \min(E(v',y_i),E(y_i,y_i')) \\
&\geq \min(E(v',y_i),(\varrho(y_i),0),(\gamma,0)) \\
&= \min(E(v',y_i),E(v',v)).
\end{align*}

		\end{enumerate}
		
	\item[$a = v$, $b = y_i'$, $c = v'$] We have three cases corresponding to (a), (b), and (c) in the definition of $E(v',v)$.

	\begin{enumerate}
		\item[(a)] Suppose that $E(v',v) = \min(E(v,y_j'),E(v,y_j))$ for some $j$. Then
\begin{align*}
E(v,y_i') &\geq \min(E(v,y_j'),E(y_j',y_i')) \\%&&\text{ by \ref{Q2}}\\
&= \min(E(v,y_j'),E(y_j,y_i)) \\
&\geq \min(E(v,y_j'),E(y_j,v),E(v,y_i)) \\%&&\text{ by \ref{Q2}}\\
&\geq \min(E(v,y_j'),E(y_j',v'),E(v',y_i')) \\
&= \min(E(v,v'),E(v',y_i')).
\end{align*}
	
		\item[(b)] Suppose that $E(v',v) = \min((\varrho(v),0),(\beta,m))$. Then
\begin{align*}
E(v,y_i') &\geq \min(E(v,y_i),E(y_i,y_i')) \\%&&\text{ by \ref{Q2}}\\
&\geq \min(E(v',y_i'),(\varrho(v),0),(\beta,m)) \\
&= \min(E(v',y_i'),E(v',v)).
\end{align*}
	
		\item[(c)] Suppose that $E(v',v) = \min((\varrho(v),0),(\gamma,0))$. As before, $E(y_i,y_i') \geq \min((\varrho(v),0),(\gamma,0))$. Then
\begin{align*}
E(v,y_i') &\geq \min(E(v,y_i),E(y_i,y_i')) \\%&&\text{ by \ref{Q2}}\\
&\geq \min(E(v',y_i'),(\varrho(v),0),(\gamma,0)) \\
&= \min(E(v',y_i'),E(v',v)).
\end{align*}

	\end{enumerate}
\end{description}
\noindent That completes the last case in the verification of \ref{Q2}.

For \ref{Q5}, we have three cases to check.
\begin{enumerate}
	\item We will show that $E_{L}(v',v) \leq E_{L}(u',u)$. We have three subcases.
	
	\begin{enumerate}
	
		\item[(a)] $E(v',v) = \min(E(v',y_i'),E(y_i',v))$ for some $i$. Then
\begin{align*}
E_{L}(u',u) &\geq \min(E_{L}(u',x_i'),E_{L}(x_i',u)) \\%&&\text{ by \ref{Q2}}\\
&= \min(E_{L}(u,x_i),E_{L}(x_i',u)) \\
& \geq \min(E_{L}(v,y_i),E_{L}(y_i',v)) \\%&&\text{ by \ref{Q5}} \\
& = \min(E_{L}(v',y_i'),E_{L}(y_i',v)) \\
&= E_{L}(v',v).
\end{align*}
	
		\item[(b)] $E_{L}(v',v) = \min(\varrho(v),\beta)$. Then $E_{L}(u,u') \geq \min(\varrho(u),\beta) \geq E_{L}(v',v)$.

		\item[(c)] $E_{L}(v',v) = \min(\varrho(v),\gamma)$. So $\beta \notin R_{|v|}$. Then, since $R_{|u|} \subseteq R_{|v|}$, $\beta \notin R_{|u|}$ and the least element of $R_{|u|}$ which is greater than $\beta$ is at least $\gamma$. Since $E_{L}(u,u') \geq \min(\varrho(u),\beta)$, $E_{L}(u,u') \geq \min(\varrho(u),\gamma) \geq E_{L}(v',v)$.

	\end{enumerate}

	\item We will show that $E_{L}(u',x_i) \geq E_{L}(v',v)$. Let $j$ be such that $E(v',y_i) = \min(E(v',y_j'),E(y_j',y_i))$. Then
\begin{align*}
E_{L}(u',x_i) &\geq \min(E_{L}(u',x_j'),E_{L}(x_j',x_i)) \\%&&\text{ by \ref{Q2}}\\
&= \min(E_{L}(u,x_j),E_{L}(x_j',x_i)) \\
& \geq \min(E_{L}(v,y_j),E_{L}(y_j',y_i)) \\%&&\text{ by \ref{Q5}} \\
&= \min(E_{L}(v',y_j'),E_{L}(y_j',y_i)) \\
&= E_{L}(v',v).
\end{align*}

	\item $E_{L}(v',y_i') = E_{L}(v,y_i) \leq E_{L}(u,x_i) = E_{L}(u',x_i')$.
	
\end{enumerate}

For \ref{Q4}, we again have three cases to check.
\begin{enumerate}
	\item Suppose that $\alpha = E_{L}(v',y_i) = E_{L}(u',x_i)$ and $n = \varepsilon(v') \geq E_{\omega}(u',x_i)$. Let $j$ be such that
	\[ E(v',y_i) = \min(E(v',y_j'),E(y_j',y_i)) = \min(E(v,y_j),E(y_j',y_i)).\]
	Thus $E(v,y_j) \geq \alpha$ and $E(y_j',y_i) \geq \alpha$. Either $E_{\omega}(x_j',x_i) > n$ or
	\[ E_{L}(x_j',x_i) > E_{L}(y_j',y_i) \geq \alpha.\]
	We always have $E_{L}(x_j',x_i) \geq \alpha$, and hence $E_{L}(x_j',x_i) > (\alpha,n)$.
	Similarly, either $E_{\omega}(u,x_j) > n$ or
	\[ E_{L}(u,x_j) > E_{L}(v,y_j) \geq \alpha\]
	and so $E_{L}(u,x_j) > (\alpha,n)$.
	Thus 
	\[ E(u',x_i) \geq \min(E(u',x_j'),E(x_j',x_i)) = \min(E(u,x_j),E(x_j',x_i)) > (\alpha,n).\]
	This is a contradiction.

	\item Suppose that $E_{L}(v',y_i') = E_{L}(u',x_i')$ and $\varepsilon(v') \geq E_\omega(u',x_i')$. We have $E(v',y_i') = E(v,y_i)$, $\varepsilon(v') = \varepsilon(v)$, and $E(u',x_i') = E(u,x_i)$. This contradicts \ref{Q4}.

	\item Suppose that $E_{L}(v',v) = E_{L}(u',u)$ and $\varepsilon(v') \geq E_{\omega}(u,u')$. We have three subcases, depending on how $E(v',v)$ gets its value.

	\begin{enumerate}

		\item[(a)] Suppose that for some $i$, 
		\[ E(v',v) = \min(E(v',y_i'),E(y_i',v)) = \min(E(v,y_i),E(y_i',v)).\]
		We have
		\[ E(u',u) \geq \min(E(u',x_i'),E(x_i',u)) = \min(E(u,x_i),E(x_i',u)).\]
		Either $E_{L}(v,y_i) < E_{L}(u,x_i)$ or $\varepsilon(v) < E_\omega(u,x_i)$,	and either $E_{L}(v,y_i') < E_{L}(u,x_i')$ or $\varepsilon(v) < E_\omega(u,x_i')$.
		
		Suppose that $E(u',u) = E(u,x_i) \leq E(x_i',u)$. The other case is similar. Then, since $E_{L}(u',u) = E_{L}(v',v)$ we must have $E_L(v,y_i) = E_L(u,x_i)$; otherwise, we would have
		\[ E_L(v',v) \leq E_L(v,y_i) < E_L(u,x_i) = E_L(u',u).\]
Hence $\varepsilon(v) < E_\omega(u,x_i)$. Thus $E_\omega(u',u) > \varepsilon(v)$, a contradiction.
		
		\item[(b)] Suppose that $E_{L}(v',v) = \min((\varrho(v),0),(\beta,m))$. Then since $E(u,u') \geq \min((\varrho(u),0),(\beta,m))$ and $\varrho(v) < \varrho(u)$, if $E_{L}(v',v) = E_{L}(u,u')$ then they are both equal to $\beta$, and $E(u,u') \geq (\beta,m)$. But $\varepsilon(v') < m$, which contradicts $\varepsilon(v') \geq E_\omega(u,u') \geq m$. 

		\item[(c)] Suppose that $E_{L}(v',v) = \min((\varrho(v),0),(\gamma,0))$. Then since $E(u,u') \geq \min((\varrho(u),0),(\beta,m))$, by choice of $\gamma$ and using the fact that $R_{|u|} \subseteq R_{|v|}$, $E(u,u') \geq \min((\varrho(u),0),(\gamma,0))$. If $E_{L}(v',v) = E_{L}(u,u')$ then they are both equal to $\gamma$. But then $\gamma \in R_{|u|}$, and so by \ref{R3}, in $R_{|v|}$ there is some $\gamma'$ with $\beta \leq \gamma' < \gamma$. This contradicts the choice of $\gamma$.

	\end{enumerate}
\end{enumerate}

For \ref{Q6}, we once more have three cases to check.
\begin{enumerate}
	\item We will show that $E_{L}(v',v) \in R_{|v|} \cup \{-\infty,\varrho(v)\}$. As usual, we have three subcases.

	\begin{enumerate}
		\item[(a)] $E_{L}(v',v) = \min(E_{L}(v',y_i'),E_{L}(y_i',v))$ for some $i$. Then each of $E_{L}(v',y_i') = E_{L}(v,y_i)$ and $E_{L}(y_i',v)$ is in $R_{|v|} \cup \{-\infty,\varrho(v)\}$ and so the same is true of $E_{L}(v',v)$.

		\item[(b)] $E_{L}(v',v) = \min((\varrho(v),0),(\beta,m))$ and $\beta \in R_{|v|}$. This is immediate.
		
		\item[(c)] $E_{L}(v',v) = \min((\varrho(v),0),(\gamma,0))$ and $\gamma \in R_{|v|}$. This is also immediate.
	\end{enumerate}

	\item  Let $j$ be such that $E(v',y_i) = \min(E(v',y_j'),E(y_j',y_i))$. Then each of $E_{L}(v',y_j') = E_{L}(v,y_j)$ and $E_{L}(y_j',y_i)$ is in $R_{|v|} \cup \{-\infty,\varrho(v)\}$, so the same is true of $E_{L}(v',y_i)$.
	
	\item $E_{L}(v',y_i') = E_{L}(v,y_i)$ which is in $R_{|v|} \cup \{-\infty,\varrho(v)\}$.
	
\end{enumerate}

We have now finished showing that the finite structure we defined above is in the class $\mathbb{K}$. So we can may assume that $v'$ is in $\mc{M}$. Note that $E(v',v) > -\infty$, so by \ref{A1}, $A_i(v') \Longleftrightarrow A_i(v)$. Thus
\[ u_1,\ldots,u_t,v \equiv_{\at} u_1',\ldots,u_t',v'. \]
By definition, we have $E(v',v) \geq \min((\varrho(v),0),(\beta,m))$. This completes the proof of the lemma.
\end{proof}

Recall that we defined an equivalence relation $\sim_\alpha$ by $x \sim_\alpha y$ if $E_{L}(x,y) \geq \min(\varrho(x),\alpha)$. We can expand this to an equivalence relation on tuples as follows.

\begin{definition}
Given $\alpha \in L$ and $x_1,\ldots,x_r$ and $x_1',\ldots,x_r'$ from $\mc{M}$ both closed under the predecessor relation $P$, define:
\[ x_{1},\ldots,x_r \sim_{\alpha} x_1',\ldots,x_r' \]
if and only if
\[ x_{1},\ldots,x_r \equiv_{\at} x_{1}',\ldots,x_r' \text{ and for all $i$, } x_i  \sim_{\alpha} x_i'.\]
If $x_1,\ldots,x_r$ and $x_1',\ldots,x_r'$ are not closed under predecessors, we can close them under predecessors in the natural way to extend $\sim_\alpha$ to a relation on all pairs of tuples.
\end{definition}

Note that $\bar{x} \sim_0 \bar{y}$ asks that $\bar{x}$ and $\bar{y}$ satisfy the same atomic formulas, whereas $\bar{x} \leq_0 \bar{y}$ asks that they satisfy the same atomic formulas with bounded G\"odel numbers. However, if we replace $\sim_0$ by $\leq_0$, these relation $\sim_\alpha$ are non-standard back-and-forth relations. Note that the relations $\sim_\alpha$ are symmetric, whereas back-and-forth relations are, a priori, not necessarily symmetric.

\begin{lemma}\label{lem:equiv}\label{lem:are-ns-bf}
$\leq_0$ and $(\sim_\alpha)_{0 < \alpha \in L}$ are non-standard back-and-forth relations in the sense of Definition \ref{def:nsbf}.
\end{lemma}
\begin{proof}
Suppose that $\alpha > 0$ and
\[ x_{1},\ldots,x_r \sim_{\alpha} x_1',\ldots,x_r'. \]
Suppose that we are given $y_1,\ldots,y_s$ and $\beta < \alpha$. We will find $y_1',\ldots,y_s'$ such that
\[ x_{1},\ldots,x_r,y_1,\ldots,y_s \sim_{\beta} x_{1}',\ldots,x_r',y_1',\ldots,y_s'.\]
We already know that
\[ x_{1},\ldots,x_r \equiv_{\at} x_1',\ldots,x_r'. \]
We may assume that the $y_i$ are closed under predecessors, and that the predecessor of each $y_i$ appears earlier in the list (or in $x_i$). Let $m \in \omega$ be large enough that for any element $z$ which we have mentioned so far (the $x_i$, $x_i'$, and $y_i$) we have $\bar{\varepsilon}(z) \in \{0,\ldots,m\}^{< \omega}$. Note that, by choice of $m$, $x_1,\ldots,x_r$ and $x_1',\ldots,x_r'$ satisfy (i) and (ii) of Lemma \ref{lem:extension} (for this $\beta$ and $m$). So using the lemma we get a $y_1'$ such that $x_1,\ldots,x_r,y_1$ and $x_1',\ldots,x_r',y_1'$ also satisfy (i) and (ii). But then we can use the lemma to get a $y_2'$, and so on, until we have $y_1',\ldots,y_s'$ as desired.

On the other hand, suppose that $\alpha > 0$ and
\[ x_{1},\ldots,x_r \nsim_{\alpha} x_1',\ldots,x_r'. \]
We need to show that there are $y_1,\ldots,y_s$ and $\beta < \alpha$ such that for all $y_1',\ldots,y_s'$,
\[ x_{1},\ldots,x_r,y_1,\ldots,y_s \nsim_{\beta} x_1',\ldots,x_r',y_1',\ldots,y_s'. \]
There are three cases.

\begin{case}\label{case:one}
$x_{1},\ldots,x_r \nequiv_{\at} x_{1}',\ldots,x_r'$.
\end{case}
\begin{proof}
There are only finitely many constant symbols from $L$ required to determine the values of all of the functions in the language on $x_{1},\ldots,x_r$, and by \ref{A2}, finitely many indices $j$ for relations $A_j$ are required to determine which of the $A_j$ hold of $x_{1},\ldots,x_r$. In particular, a finite set of formulas from the language determines the entire atomic diagram of $x_{1},\ldots,x_r$. Hence, for any arbitrary choice of $y_1,\ldots,y_s$ with $s$ an upper bound on the G\"odel numbers of those finitely many formulas, for all $y_1',\ldots,y_s'$,
\[ x_{1},\ldots,x_r,y_1,\ldots,y_s \nleq_{0} x_{1}',\ldots,x_r',y_1',\ldots,y_s'.\qedhere\]
\end{proof}

In the other two cases, we have
\[ x_{1},\ldots,x_r \equiv_{\at} x_{1}',\ldots,x_r'. \]
It follows (by \ref{A1}) that $E(x_i,x_i') > -\infty$ for each $i$. Also, since
\[ x_{1},\ldots,x_r \nsim_{\alpha} x_1',\ldots,x_r' \]
there is some $i$ such that $E_{L}(x_i,x_i') < \min(\varrho(x_i),\alpha)$.

\begin{case} $x_{1},\ldots,x_r \equiv_{\at} x_{1}',\ldots,x_r'$ and for some $i$, $E_{L}(x_i,x_i') < \varrho(x_i) < \alpha$.
\end{case}
\begin{proof}
We have
\[ x_{1},\ldots,x_r \nsim_{\varrho(x_i)} x_1',\ldots,x_r'. \]
Since $\varrho(x_i) < \alpha$, we are done in this case.
\end{proof}

\begin{case} $x_{1},\ldots,x_r \equiv_{\at} x_{1}',\ldots,x_r'$ and for some $i$, $E_{L}(x_i,x_i') < \alpha \leq \varrho(x_i)$.
\end{case}
\begin{proof}
Recall that $E(x_i,x_i') > -\infty$. Let $\beta = E_{L}(x_i,x_i')$ and $\ell = E_{\omega}(x_i,x_i')$. There is a successor $y$ of $x_i$ with $\varrho(y) = \beta < \varrho(x_i)$ and $\varepsilon(y) \geq \ell$. By \ref{Q5} and \ref{Q4}, for all $y'$ successors of $x_i'$, $E_{L}(y,y') < \beta$. Then
\[ x_{1},\ldots,x_r,y \nsim_{\beta} x_{1}',\ldots,x_r',y' \]
for all $y'$. If $\beta > 0$, then we are already done.
Otherwise, if $\beta = 0$ is the least element of $L$, then by \ref{A1} for all $y'$ we have
\[ x_{1},\ldots,x_r,y \nequiv_{\at} x_{1}',\ldots,x_r',y'. \]
As in Case \ref{case:one} for any arbitrary choice of $y_1,\ldots,y_s$ with $s$ sufficiently large, for all $y',y_1',\ldots,y_s'$,
\[ x_{1},\ldots,x_r,y,y_1,\ldots,y_s \nleq_{0} x_{1}',\ldots,x_r',y',y_1',\ldots,y_s'. \qedhere\]
\end{proof}

These three cases end the proof of the lemma.
\end{proof}

\begin{corollary}\label{lem:b-f-leq-same}
For $\alpha \in \wf(O)$, $\alpha > 0$, and $\bar{a},\bar{b} \in \mc{M}$, the following are equivalent:
\begin{enumerate}
	\item $\bar{a} \sim_\alpha \bar{b}$,
	\item $\bar{a} \geq_\alpha \bar{b}$,
	\item $\bar{a} \leq_\alpha \bar{b}$.
\end{enumerate}
\end{corollary}

\begin{lemma}\label{lem:aut-orbit}
Let $\mc{M} \models T(L)$. Then there is an automorphism of $\mc{M}$ taking $\bar{a}$ to $\bar{a}'$ if and only if $\bar{a} \sim_\alpha \bar{a}'$ for all $\alpha \in \wf(L)$.
\end{lemma}
\begin{proof}
If, for some $\alpha \in \wf(L)$, $\bar{a} \nsim_\alpha \bar{a}'$, then by Lemma \ref{lem:b-f-leq-same} $\bar{a} \nleq_\alpha \bar{a}'$. So there is no automorphism taking $\bar{a}$ to $\bar{a}'$.

On the other hand, suppose that for all $\alpha \in \wf(L)$, $\bar{a} \sim_\alpha \bar{a}'$. We have three cases.

\begin{case}
$L$ is well-founded and has a maximal element.
\end{case}

Let $\alpha$ be the maximal element of $L$. We claim that the set of finite partial maps
\[ \{ \bar{a} \mapsto \bar{a}' : \bar{a} \sim_\alpha \bar{a}' \} \]
has the back-and-forth property. It suffices to assume that $\bar{a}$ and $\bar{a}'$ are closed under predecessors. It also suffices to check the back-and-froth property for adding an element $b$ which is a child of one of the $a_i$. Since $\alpha$ is the maximal element of $L$, $x \sim_\alpha y$ if and only if $E(x,y) = \varrho(x)$ for each $i$.

Let $\bar{a} = (a_1,\ldots,a_t)$ and $\bar{a}'=(a_1',\ldots,a_t')$ be such that $\bar{a} \sim_\alpha \bar{a}'$. Note that $\bar{a} \equiv_{\at} \bar{a}'$ and for each $i$, $E(a_i,a_i') = (\varrho(a_i),0)$. Given $b$ a child of $a_i$, by Lemma \ref{lem:extension} (with $\beta = \alpha$ and $m = 0$) there is $b'$ such that $\bar{a},b \equiv_{\at} \bar{a}',b'$ and $E(b,b') = (\varrho(b),0)$. Hence $\bar{a},b \sim_\alpha \bar{a}',b'$.

\begin{case}
$L$ is well-founded and has no maximal element.
\end{case}

We claim that the set of finite partial maps
\[ \{ \bar{a} \mapsto \bar{a}' : \bar{a} \sim_\alpha \bar{a}' \text{ for all $\alpha \in L$} \} \]
has the back-and-forth property. It suffices to assume that $\bar{a}$ and $\bar{a}'$ are closed under predecessors. It also suffices to check the back-and-forth property for adding an element $b$ which is a child of one of the $a_i$.

Let $\bar{a} = (a_1,\ldots,a_t)$ and $\bar{a}'=(a_1',\ldots,a_t')$ be such that $\bar{a} \sim_\alpha \bar{a}'$ for all $\alpha \in L$. Note that $\bar{a} \equiv_{\at} \bar{a}'$ and for each $i$, $E(a_i,a_i') = (\varrho(a_i),0)$. Given $b$ a child of $a_i$, let $\beta$ be such that $\beta > \varrho(b),\varrho(a_1),\ldots,\varrho(a_t)$. Then by Lemma \ref{lem:extension} (with this $\beta$ and $m = 0$), there is $b'$ such that $\bar{a},b \equiv_{\at} \bar{a}',b'$ and $E(b,b') = (\varrho(b),0)$, and hence $\bar{a},b \sim_\alpha \bar{a}',b'$ for all $\alpha \in L$.

\begin{case}
$L$ is not well-founded.
\end{case}

Let $\bar{a} = (a_1,\ldots,a_t)$ and $\bar{a}'=(a_1',\ldots,a_t')$ be such that $\bar{a} \sim_\alpha \bar{a}'$ for all $\alpha \in L$. We claim that there is $\alpha \notin \wf(L)$ such that $\bar{a} \sim_\alpha \bar{a}'$. Then, by Lemmas \ref{lem:are-ns-bf} and \ref{lem:non-standard-bf}, we would get an automorphism of $\mc{M}$ taking $\bar{a}$ to $\bar{a}'$.

We claim that for each $i$, either $\varrho(a_i) \in \wf(L)$ or $E(a_i,a_i') \notin \wf(L)$. This is enough to get $\bar{a} \sim_\alpha \bar{a}'$ for some $\alpha \notin \wf(L)$. If $\varrho(a_i) \notin \wf(L)$, then since $a_i \sim_\alpha a_i'$ for all $\alpha \in \wf(L)$, $E(a_i,a_i') \geq \alpha$ for all $\alpha \in \wf(L)$. By \ref{O2}, $E(a_i,a_i') \notin \wf(L)$.
\end{proof}

\begin{lemma}\label{lem:def-sim}
Given $\bar{x}$ a tuple in $\mc{M}$ and $\alpha \in \wf(L)$, there is a $\Pi^{\infi}_{\alpha}$ formula which defines the set of $\bar{y}$ with $\bar{x} \sim_\alpha \bar{y}$.
\end{lemma}
\begin{proof}
Let $\bar{y}$ be such that $\bar{x} \nsim_\alpha \bar{y}$. By Lemma \ref{lem:b-f-leq-same}, $\bar{x} \nleq_{\alpha} \bar{y}$. Proposition 15.1 of \cite{AshKnight00} says that $\bar{x} \leq_{\alpha} \bar{y}$ if and only if every $\Sigma_\alpha^{\infi}$ formula true of $\bar{y}$ is true of $\bar{x}$. So there is a $\Sigma^{\infi}_\alpha$ formula $\varphi_{\bar{y}}$ true of $\bar{y}$ but not of $\bar{x}$. Let
\[ \psi = \bigdoublewedge_{\bar{x} \nsim_\alpha \bar{y}} \neg \varphi_{\bar{y}}.\]
Note that $\psi$ is a $\Pi^{\infi}_\alpha$ formula. If $\mc{M} \models \neg\psi(\bar{z})$ then there is $\bar{y}$ such that $\mc{M} \models \varphi_{\bar{y}}(\bar{z})$ and so $\bar{x} \nleq_{\alpha} \bar{z}$ (and hence $\bar{x} \nsim_{\alpha} \bar{z}$). On the other hand, if $\bar{x} \nsim_{\alpha} \bar{z}$, then $\mc{M} \models \varphi_{\bar{z}}(\bar{z})$ and so $\mc{M} \models \neg \psi(\bar{z})$.
\end{proof}

\subsection{Computation of Scott Rank for Theorem \ref{thm:construction}}

Recall that for Theorem \ref{thm:construction}, we add to $S^+$:
\begin{enumerate}[label=(O\arabic*)]
	\item[(O4a)] $R_n = L$ for all $n$.
\end{enumerate}
So \ref{Q6} is a vacuous condition. The following lemma completes the proof of Theorem \ref{thm:construction}:

\begin{lemma}\label{lem:sr-comp}
Let $\mc{M} \models T(L)$. Then $SR(\mc{M}) = \wfc(L)$.
\end{lemma}
\begin{proof}
Recall Theorem \ref{thm:montalban}, which says that the Scott rank of $\mc{M}$ is the least $\alpha$ such that every automorphism orbit is $\Sigma^{\infi}_\alpha$-definable without parameters. We have two cases.

\begin{case}
$L$ is well-founded.
\end{case}

Let $\bar{x} = (x_1,\ldots,x_n)$ be a tuple in $\mc{M}$. Let $\alpha \in L$ be such that $\alpha \geq \varrho(x_1),\ldots,\varrho(x_n)$. Then for $\gamma \geq \alpha$ and $\bar{y} \in \mc{M}$, $\bar{x} \sim_{\gamma} \bar{y}$ if and only if $\bar{x} \sim_{\alpha} \bar{y}$. So, by Lemma \ref{lem:aut-orbit}, $\bar{x} \sim_{\alpha} \bar{y}$ if and only if $\bar{x}$ and $\bar{y}$ are in the same automorphism orbit. By Lemma \ref{lem:def-sim}, the orbit of $\bar{x}$ is $\Pi^{\infi}_\alpha$-definable. Thus the orbits of all of the tuples $\bar{x}$ from $\mc{M}$ are $\Sigma^{\infi}_{\wf(L)}$-definable.

Let $\alpha \in L$, $\alpha > 0$. By Lemma \ref{lem:third-fraisse} there is $x \in \mc{M}$ a successor of $\la \ra$ with $\varrho(x) = \alpha$. We claim that the automorphism orbit of $x$ is not definable by a $\Sigma^{\infi}_\alpha$ formula. Suppose to the contrary that it was, say by a formula $\varphi$. Let $\bar{y} = (y_1,\ldots,y_s)$ be a tuple in $\mc{M}$ and $\psi$ a $\Pi^{\infi}_\beta$ formula for some $\beta < \alpha$ which witness that $\mc{M} \models \varphi(x)$. Let $m \in \omega$ be such that $m > \varepsilon(y_1),\ldots,\varepsilon(y_s)$. Using the construction of $\mc{M}$ as a Fra\"{\i}ss\'e limit, there is $x' \in \mc{M}$ such that $(\alpha,0) > E(x,x') > (\beta,m)$. So $x \nsim_\alpha x'$. By Lemma \ref{lem:extension} with these values of $\beta$ and $m$, there is a tuple $\bar{y}'=(y_1',\ldots,y_s')$ such that $x,\bar{y} \sim_\beta x',\bar{y}'$. By Lemma \ref{lem:b-f-leq-same}, $x,\bar{y} \leq_\beta x',\bar{y}'$. Hence $\mc{M} \models \varphi(x')$. So $x$ and $x'$ are in the same automorphism orbit; but then Lemma \ref{lem:aut-orbit} contradicts the fact that $x \nsim_\alpha x'$.

So the automorphism orbits of $\mc{M}$ are definable by $\Sigma^{\infi}_{\wf(L)}$ formulas, but there is no $\alpha < \wf(L)$ such that all of the automorphism orbits are definable by $\Sigma^{\infi}_{\alpha}$ formulas. So $SR(\mc{M}) = \wf(L) = \wfc(L)$ since $L$ is well-founded.

\begin{case}
$L$ is not well-founded.
\end{case}

Fix a tuple $\bar{x}$. By Lemma \ref{lem:aut-orbit}, $\bar{y}$ is in the automorphism orbit of $\bar{x}$ if and only if $\bar{x} \sim_\alpha \bar{y}$ for each $\alpha \in \wf(L)$. By Lemma \ref{lem:def-sim}, the set of such $\bar{y}$ is $\Pi^{\infi}_\alpha$-definable for each fixed $\alpha$. So the set of $\bar{y}$ for which $\bar{x} \sim_\alpha \bar{y}$ for all $\alpha \in \wf(L)$ is $\Pi^{\infi}_{\wf(L)}$-definable, and therefore $\Sigma^{\infi}_{\wf(L) + 1}$-definable.

By Lemma \ref{lem:third-fraisse} there is $x \in \mc{M}$ a successor of $\la \ra$ with $\varrho(x) = \alpha \notin \wf(L)$. The argument from the previous case shows that the automorphism orbit of $x$ is not definable by a $\Sigma^{\infi}_{\beta}$ formula for any $\beta \in \wf(L)$. Hence it is not definable by a $\Sigma^{\infi}_{\wf(L)}$ formula since $\wf(L)$ is a limit ordinal by \ref{O2}.

We have shown that the automorphism orbits of $\mc{M}$ are all definable by $\Sigma^{\infi}_{\wf(L) +1}$ formulas, but that not all of the automorphism orbits are definable by $\Sigma^{\infi}_{\wf(L)}$ formulas. Since $L$ is not well-founded, $\wf(L) + 1 = \wfc(L)$. So $SR(\mc{M}) = \wfc(L)$. 
\end{proof}

\subsection{Computation of Scott Rank for Theorem \ref{thm:construction2}}\label{subsec:omegaCK}

For Theorem \ref{thm:construction2}, we want to have: If $\mc{M} \models T(L)$, then $SR(\mc{M}) = \wf(L)$ (rather than $\wfc(L)$). We accomplish this by adding to $S^+$:
\begin{enumerate}[label=(O4b)]
	\item\label{O4b} There is a function $G \colon L \to \mc{P}(L)$ such that for each $\alpha \in L$, $G(\alpha)$ is an increasing sequence of order type $\omega$ whose limit is $\alpha$ if $\alpha$ is a limit ordinal, or a finite set containing $\alpha - 1$ if $\alpha$ is a successor ordinal. $R_1$ is an increasing sequence with order type $\omega$ which is unbounded in $L$, or $R_1$ is $\{\gamma\}$ if $\gamma$ is the maximal element of $L$. For each $n$,
\[ R_{n+1} = \{ \beta_n \} \cup R_n \cup \bigcup_{\alpha \in R_n} G(\alpha) \]
where $\beta_n$ is an element of $L$. If $\gamma < \alpha < \beta$, and $\gamma \in G(\beta)$, then $\gamma \in G(\alpha)$.
\end{enumerate}
In this case, \ref{Q6} is no longer a vacuous condition. Recall that while \ref{O4b} is not $\Pi^{\infi}_2$, this does not matter as we can take its Morleyization.

\ref{O4b} plays a similar role to the ``thin trees'' in Knight and Millar's \cite{KnightMillar10} construction of a computable structure of Scott rank $\omega_1^{CK}$; the ordinals in $R_n$ put a bound, at level $n$ of the tree, on the Scott ranks of the elements at that level. In Lemma \ref{lem:prop-of-Rn} below, we will see that for each $n$ there is a bound, below $\wf(L)$, on the ordinals in $R_n \cap \wf(L)$.

We begin by showing that \ref{O4b} implies \ref{O3}. \ref{R1}, \ref{R2}, \ref{R3}, and \ref{R4} follow immediately from \ref{O4b}. To see \ref{R5}, we show in the following lemma that the $R_n$ are well-founded.

\begin{lemma}\label{lem:prop-of-Rn}
Suppose that $G$ and $R_n$ satisfy \ref{O4b} and that $L$ is not well-founded. For each $n$, $R_n$ is well-founded and there is a bound $\alpha_n \in \wf(L)$ on $R_n \cap \wf(L)$.
\end{lemma}
\begin{proof}
We argue by simultaneous induction that:
\begin{enumerate}
	\item for each $n$, $R_n$ is well-founded,
	\item for each $n$ there is a bound $\alpha_n \in \wf(L)$ on $R_n \cap \wf(L)$.
\end{enumerate}

For $n = 1$, $R_1$ has order type at most $\omega$ and is unbounded in $L$. Thus $R_1$ is well-founded, and has only finitely many elements in $\wf(L)$.

For the case $n+1$, we have that
\[ R_{n+1} = \{ \beta_n \} \cup R_n \cup \bigcup_{\alpha \in R_n} G(\alpha). \]
We claim that $R_{n+1}$ is well-ordered. It suffices to show that $\bigcup_{\alpha \in R_n} G(\alpha)$ is well-ordered. Suppose that $(\beta_n)_{n \in \omega}$ is a decreasing sequence in $\bigcup_{\alpha \in R_n} G(\alpha)$. For each $n$, let $\alpha_n$ be the least $\alpha \in R_n$ such that $\beta_n \in G(\alpha_n)$. We claim that $(\alpha_n)_{n \in \omega}$ is a non-increasing sequence. If not, then for some $n$, $\alpha_{n+1} > \alpha_n$. We have $\alpha_{n+1} > \alpha_n \geq \beta_n \geq \beta_{n+1}$ and $\beta_{n+1} \in G(\alpha_{n+1})$, so by \ref{O4b} we have $\beta_{n+1} \in G(\alpha_n)$. This contradicts the choice of $\alpha_{n+1}$. Since $R_n$ is well-founded, there is $\alpha$ such that for all sufficiently large $n$, $\alpha_n = \alpha$. Thus, throwing away some initial part of the sequence $(\beta_n)_{n \in \omega}$, we may assume that each $\beta_n \in G(\alpha)$. This is a contradiction, since $G(\alpha)$ has order type at most $\omega$.

Now we claim that there is a bound $\alpha_{n+1} \in \wf(L)$ on $R_{n+1} \cap \wf(L)$. Let $\gamma$ be the least element of $R_{n}$ which is not in $\wf(L)$. Then we claim that
\[ R_{n+1} \cap \wf(L) = \{ \beta_{n+1} \} \cup [R_n \cap \wf(L)] \cup [G(\gamma) \cap \wf(L)] \cup \bigcup_{\alpha \in R_n \cap \wf(L)} G(\alpha). \]
(If $\beta_{n+1} \notin \wfc(L)$, then we omit it.) Clearly the right hand side is contained in the left hand side. To see that we have equality, suppose that $\alpha \in \wf(L) \cap G(\delta)$ for some $\delta \in R_n \setminus \wf(L)$. Then $\alpha < \gamma \leq \delta$, and so $\alpha \in G(\gamma)$ by \ref{O4b}.

Then since $G(\gamma)$ has order type at most $\omega$ and $\gamma \notin \wf(L)$, $G(\gamma) \cap \wf(L)$ is finite and hence bounded. Also, each element of $G(\alpha)$ for $\alpha \in R_n \cap L$ is bounded above by $\alpha$, and since $R_n \cap L$ is bounded above by $\alpha_n$, $\bigcup_{\alpha \in R_n \cap \wf(L)} G(\alpha)$ is also bounded above by $\alpha_n$. Thus $R_{n+1} \cap \wf(L)$ is bounded above by some $\alpha_{n+1} \in \wf(L)$.
\end{proof}

We also need to know that for each $L \in \mc{S}$, there are $G$ and $R_n$ satisfying \ref{O4b}. Moreover, if $L$ is computable with a computable successor relation, then we need to find $G$ and $R_n$ computable in order to have a computable model of $T(L)$.

\begin{lemma}
There are $G$ and $R_1$ satisfying \ref{O4b} such that $L = \bigcup_n R_n$. If $L$ is computable with a computable successor relation, then we can pick $G$, $(\beta_n)_{n \in \omega}$, and $R_1$ such that the $R_n$ are uniformly computable.
\end{lemma}
\begin{proof}
Let $\alpha_0,\alpha_1,\alpha_2,\ldots$ be a listing of $L$. To define $R_1$, greedily pick an increasing subsequence of $\alpha_0,\alpha_1,\alpha_2,\ldots$ (or, if there is a maximal element $\gamma$ of $L$, let $R_1 = \{ \gamma \}$).

We begin by defining $G(\alpha_0)$, after which we define $G(\alpha_1)$, and so on. To define $G(\alpha_0)$, greedily pick an increasing subsequence of $\alpha_1,\alpha_2,\ldots$ each element of which is less than $\alpha_0$ (stopping if we ever find a predecessor of $\alpha_0$). Now suppose that we have defined $G(\alpha_0),\ldots,G(\alpha_n)$. Suppose that $\beta_1,\ldots,\beta_\ell$ are those $\alpha_i$, $0 \leq i \leq n$, with $\alpha_i < \alpha_{n+1}$. Let $\beta = \max(\beta_1,\ldots,\beta_\ell)$. Let $\gamma_1,\ldots,\gamma_m$ be those $\alpha_i$, $0 \leq i \leq n$, with $\alpha_i > \alpha_{n+1}$.  Begin by putting into $G(\alpha_{n+1})$ the finitely many elements of $G(\gamma_1),\ldots,G(\gamma_m)$ which are less than $\alpha_{n+1}$. If $\alpha_{n+1}$ is the successor of one of these elements, then we are done; if $\alpha_{n+1}$ is the successor of one of $\alpha_1,\ldots,\alpha_n$, then add that element to $G(\alpha_{n+1})$. Otherwise, greedily pick an increasing subsequence of $\alpha_{n+2},\alpha_{n+3},\ldots$, each element of which is at least $\beta$ and less than $\alpha_{n+1}$. If we ever find a predecessor of $\alpha_{n+1}$, then we can stop there. It is easy to check that $G$ is as required by \ref{O4b}. In particular, if $\gamma < \alpha < \beta$, and $\gamma \in G(\beta)$, then $\gamma \in G(\alpha)$.

Now for each $n$, let
\[ R_{n+1} = \{\alpha_n\} \cup R_n \cup \bigcup_{\alpha \in R_n} R(\alpha).\]
Note that $L = \bigcup_n R_n$.

We claim that if $\alpha_0,\alpha_1,\ldots$ was an effective listing, then each $G(\alpha_i)$ and each $R_n$ is computable. First, note that $R_1$ is computable. It is also easy to see that each $G(\alpha_n)$ is computable. Given a way of computing $R_n$, we will show that to compute $R_{n+1}$. To check whether $\alpha_i \in R_{n+1}$, first check whether $i = n$ (as we know that $\alpha_n \in R_{n+1}$). Second, check whether $\alpha_i \in R_n$. Third, check whether $\alpha_i + 1$ is in $R_n$; if it is, then $\alpha_i \in G(\alpha_i + 1) \subseteq R_{n+1}$. Fourth, check whether $\alpha_i$ is in one of $G(\alpha_0),\ldots,G(\alpha_{i-1})$. Note from the construction above that if $\alpha_i$ is not in one of these sets (and $\alpha_i$ is not in $R_{n+1}$ for one of the first three reasons), then it is not in $G(\alpha_j)$ for any $j$ with $\alpha_j \in R_n$.
\end{proof}

Finally, we show that models of $T(L)$ have the correct Scott rank.

\begin{lemma}
Let $\mc{M} \models T(L)$. Then $SR(\mc{M}) = \wfc(L)$.
\end{lemma}
\begin{proof}
Recall Theorem \ref{thm:montalban}, which says that the Scott rank of $\mc{M}$ is the least $\alpha$ such that every automorphism orbit is $\Sigma^{\infi}_\alpha$-definable without parameters. We have two cases.

\begin{case}
$L$ is well-founded.
\end{case}

This case is the same as the corresponding case of Lemma \ref{lem:sr-comp}.

\begin{case}
$L$ is not well-founded.
\end{case}

Fix a tuple $\bar{x}$. By Lemma \ref{lem:aut-orbit}, $\bar{y}$ is in the automorphism orbit of $\bar{x}$ if and only if $\bar{x} \sim_\alpha \bar{y}$ for each $\alpha \in \wf(L)$. By Lemma \ref{lem:prop-of-Rn}, there is a bound $\gamma \in \wf(L)$ such that $\bar{y}$ is in the same orbit as $\bar{x}$ if and only if $\bar{x} \sim_\alpha \bar{y}$ for each $\alpha \leq \gamma$. By Lemma \ref{lem:def-sim}, the set of such $\bar{y}$ is $\Pi^{\infi}_\alpha$-definable for each fixed $\alpha$. So the set of $\bar{y}$ for which $\bar{x} \sim_\alpha \bar{y}$ for all $\alpha \leq \gamma$ is $\Pi^{\infi}_{\gamma}$-definable.

Let $\alpha \in \wf(L)$ and by \ref{O1} let $n$ be such that $\alpha \in R_n$. By Lemma \ref{lem:third-fraisse} there is $x \in \mc{M}$, $|x| = n$, with $\varrho(x) = \alpha \notin \wf(L)$. The argument from the previous case shows that the automorphism orbit of $x$ is not definable by a $\Sigma^{\infi}_{\beta}$ formula for any $\beta \leq \alpha$. 

We have shown that the automorphism orbits of $\mc{M}$ are all definable by $\Sigma^{\infi}_{\wf(L)}$ formulas, but that there is no $\alpha \in \wf(L)$ such that all of the automorphism orbits are definable by $\Sigma^{\infi}_{\alpha}$ formulas. So $SR(\mc{M}) = \wf(L)$.
\end{proof}

\section{\texorpdfstring{$\Pi^0_2$ Theories}{Pi Two Theories}}\label{sec:pi-two}
Recall that Theorem \ref{thm:pi-two} stated that given $\alpha < \omega_1$, there is a $\Pi^{\infi}_2$ sentence $T$ all of whose models have Scott rank $\alpha$. This theorem is a simple application of the main construction.

\begin{proof}[Proof of Theorem \ref{thm:pi-two}]
Let $\mc{A} = (A,<_{\mc{A}})$ be a presentation of $(\alpha,<)$ as a structure with domain $A = (a_i)_{i \in \omega}$. Consider the atomic diagram of $\mc{A}$ in the language with constant symbols $(\underline{a}_i)_{i \in \omega}$ for the elements of $A$. Let $S$ be the conjunction of all of the sentences in the atomic diagram of $\mc{A}$, together with the sentence $(\forall x) \bigdoublevee_i (x = \underline{a}_i)$. Let $T$ be the $\Pi^{\infi}_2$ sentence obtained from Theorem \ref{thm:construction}. Then
\[ \scs(T) = \{ \alpha \}. \qedhere \]
\end{proof}

In the main construction, the sentence we built had uncountably many existential types. This was necessary: an omitting types argument shows that if a $\Pi^{\infi}_2$ sentence has only countably many existential types, then it must have a model of Scott rank 1.

\section{Computable Models of High Scott Rank}\label{sec:comp}
In this section, we will prove Theorem \ref{thm:comp} by producing a $\Pi^{\comp}_2$ sentence $T$ all of whose models have Scott rank $\omega_1^{CK} + 1$ (and a $\Pi^{\comp}_2$ sentence whose models all have Scott rank $\omega_1^{CK}$). Moreover $T$ will have a computable model. If $\mc{A}$ is a model of this sentence, then whenever $\mc{B}$ is another structure with $\mc{A} \equiv_2 \mc{B}$, $\mc{B}$ will also be a model of $T$ and hence will also have non-computable Scott rank. Thus it is not the case that every computable structure $\mc{A}$ of high Scott rank is approximated by models of lower Scott rank in the sense that for each $\alpha < \omega_1^{CK}$, there is a structure $\mc{B}_{\alpha}$ with $SR(\mc{B}_{\alpha}) < \omega_1^{CK}$ such that $\mc{A} \equiv_\alpha \mc{B}_{\alpha}$.

\begin{proof}[Proof of Theorem \ref{thm:comp}]
Let $\mc{H} = (H,<_{\mc{H}})$ be a computable presentation of the Harrison linear ordering of order type $\omega_1^{CK} \cdot (1 + \mathbb{Q})$ as a structure with domain $H = (h_i)_{i \in \omega}$. We may assume that the successor relation is computable by replacing each element of $\mc{H}$ by $\omega$ (this does not change the order type). Let $S$ be the conjunction of the sentences of the atomic diagram of $\mc{H}$ (in the language with constants $\underline{h}_i$) together with the sentence $(\forall x) \bigdoublevee_i (x = \underline{h}_i)$. Let $T$ be the $\Pi^{\comp}_2$ sentence obtained from Theorem \ref{thm:construction} applied to $S$. Then
\[ \{ SR(\mc{M}) : \mc{M} \models T \} = \{ \wfc(L) : L \models S \} = \{ \omega_1^{CK} + 1 \}.\]
To get Scott rank $\omega_1^{CK}$, we simply use Theorem \ref{thm:construction2} instead of Theorem \ref{thm:construction}.
\end{proof}

A similar argument also gives the following:
\begin{theorem}
Let $\alpha$ be a computable ordinal. There is a $\Pi^{\comp}_2$ sentence with a computable model whose computable models all have Scott rank $\alpha$.
\end{theorem}
\begin{proof}
The proof is the same as that of the previous theorem, using the fact that every computable ordinal has a presentation where the successor relation is computable \cite{Ash86,Ash87,AshKnight00}.
\end{proof}

\section{A Technical Lemma}
The general construction of Section \ref{sec:gen-const} references $PC_{\mc{L}_{\omega_1 \omega}}$-classes of linear orders, whereas our classification in Theorems \ref{thm:main-result} and \ref{thm:main-result2} references $\bfSigma^1_1$ classes of linear orders. The following technical lemma shows that, if we are only interested in the order types represented in the class, the two are equivalent.

\begin{lemma}
Let $\mc{C}$ be a class of linear orders (i.e., of order types). The following are equivalent:
\begin{enumerate}
	\item $\mc{C}$ is the set of order types of a $\bfSigma_1^1$ class of linear orders.
	\item $\mc{C}$ is the set of order types of a $PC_{\mc{L}_{\omega_1 \omega}}$-class of linear orders.
\end{enumerate}
Moreover, the lightface notions are also equivalent:
\begin{enumerate}
	\item $\mc{C}$ is the set of order types of a (lightface) $\Sigma_1^1$ class of linear orders.
	\item $\mc{C}$ is the set of order types of a computable $PC_{\mc{L}_{\omega_1 \omega}}$-class of linear orders.
\end{enumerate}
\end{lemma}
\begin{proof}
(2)$\Rightarrow$(1) is clear. For (1)$\Rightarrow$(2), suppose that $\mc{C}$ is a class of linear orders defined by $\exists X \varphi(X,\leq)$ where $\varphi$ has only quantifiers over $\omega$. Consider the class $\mc{C}^{+}$ of pairs $(\leq,X) \subseteq \omega^2 \times \omega$ with $\varphi(X,\leq)$. Then $\mc{C}^{+}$ is a Borel class.

Let $\mc{D}$ be the class of models in the language $\{\preceq, Y\} \cup \{a_i : i \in \omega\}$ such that each element of the domain is named by a unique constant $a_i$, and such that with $\leq \subseteq \omega^2$ defined by $i \leq j \Longleftrightarrow a_i \preceq a_j$ and $X \subseteq \omega$ defined by $i \in X \Longleftrightarrow a_i \in Y$, $(\leq,X) \in \mc{C}^+$. This gives a Borel reduction from $\mc{D}$ to $\mc{C}^{+}$, and hence $\mc{D}$ is Borel. By a theorem of Lopez-Escobar \cite{Lopez65}, $\mc{D}$ is $\mc{L}_{\omega_1 \omega}$-axiomatizable since it is closed under isomorphism. Moreover, the order types of models in $\mc{D}$ are the same as the order types of the linear orders in $\mc{C}^+$ and hence  the same as those in $\mc{C}$.

For the lightface notions, the proof is the same except that we use Vanden Boom's \cite{VandenBoom07} lightface analogue of the Lopez-Escobar theorem.
\end{proof}

\section{Bounds on Scott Height}\label{sec:sh}
Recall that Sacks \cite{Sacks83} showed that $\sh(\mc{L}^{\comp}_{\omega_1,\omega}) \leq \delta^1_2$ and that Marker \cite{Marker90} showed that $\sh(PC_{\mc{L}^{\comp}_{\omega_1 \omega}}) = \delta^1_2$. We now show that $\sh(\mc{L}^{\comp}_{\omega_1,\omega}) = \delta^1_2$.

\begin{proof}[Proof of Theorem \ref{thm:sh}]
Fix $\alpha < \delta^1_2$. We may assume that $\alpha \geq \omega_1^{CK}$. Let $\mc{S}$ be a computable $PC_{\mc{L}_{\omega_1 \omega}}$-class in a language $\mc{L}$ with $\alpha < \sh(\mc{S}) < \delta^1_2$. When we say that $\mc{S}$ is a computable class, we mean that there is a computable $\mc{L}_{\omega_1 \omega}$-sentence $T$ in a language $\mc{L}' \supseteq \mc{L}$ such that $\mc{S}$ is the class of reducts of models of $T$ to $\mc{L}$.

We define a (lightface) $\Sigma^1_1$ class $\mc{C}$ of linear orders as follows. $(L,\leq)$ is in $\mc{C}$ if and only if there is:
\begin{enumerate}
	\item a model $\mc{A}$ of $T$,
	\item a set $X \subseteq \omega$,
	\item a Harrison linear order $\mc{H}$ relative to $X$,
	\item an embedding $f \colon L \to \mc{H}$ such that $f(L)$ is an initial segment of $\mc{H}$, and
	\item non-standard back-and-forth relations $\precsim_\alpha$ on $\mc{A}$ indexed by $L$ (in the language $\mc{L}$ of $\mc{S}$, not of $T$), such that:
\begin{enumerate}
	\item\label{S1} for all $\alpha \in L$, there is $\bar{x}$ which is $\alpha$-free, i.e. for all $\bar{y}$ and $\beta < \alpha$, there are $\bar{x}'$ and $\bar{y}'$ such that $\bar{x}, \bar{y} \precsim_\beta \bar{x}', \bar{y}'$ and $\bar{x}' \nprecsim_\alpha \bar{x}$,
	\item\label{S2} the set of partial maps
	\[ \{ \bar{a} \mapsto \bar{b} : \bar{a} \precsim_{\alpha} \bar{b} \text{ for all $\alpha$}\} \]
	has the back-and-forth property.
\end{enumerate}
\end{enumerate}

If $\mc{A} \models T$ has Scott rank $\alpha$ (where we compute Scott rank in the language $\mc{L}$ of $\mc{S}$), then take $X$ such that $\omega_1^X > \alpha$. Let $\mc{H}$ be a Harrison linear order relative to $X$. Then $\alpha$ embeds into an initial segment of $\mc{H}$, and we can take the standard back-and-forth relations on $\mc{A}$, indexed by elements of $\alpha$. By Theorem \ref{thm:montalban}, (\ref{S1}) and (\ref{S2}) are satisfied. Thus $\alpha \in \mc{C}$.

On the other hand, if $(L,\leq)$ is well-founded, and there is a model $\mc{A}$ of $T$ with back-and-forth relations indexed by $L$ satisfying (\ref{S1}) and (\ref{S2}), then $L$ is the Scott rank of $\mc{A}$. Thus $\mc{C} \cap \On = \scs(T)$.

We claim that if $L \in \mc{C}$, then $\wf(L) \leq \sup(\scs(\mc{S}))$. If $L$ is well-founded, then this is clear, so assume that $L$ is not well-founded. Thus, for some set $X$, $L$ embeds into a Harrison linear order $\mc{H}$ relative to $X$ as an initial segment (and there is a model $\mc{A}$ of $T$ and non-standard back-and-forth relations as above). Since $L$ is not well-founded, its image in $\mc{H}$ includes the well-founded part $\omega_1^X$. Now, for each $\alpha \in \wf(L)$, by (\ref{S1}) there is $\bar{x}$ which is $\alpha$-free. Thus the Scott rank of $\mc{A}$ is at least $\omega_1^X$. Hence $\wf(L) \leq SR(\mc{A})$, and since $L$ was arbitrary, $\wf(L) \leq \sup(\scs(\mc{S}))$.

By the lightface version of Theorem \ref{thm:construction2}, there is a computable $\mc{L}_{\omega_1 \omega}$-sentence $T'$ such that
\[ \scs(T') = \{ \wf(L) \colon L \in \mc{C}\} \supseteq \mc{C} \cap \On = \scs(\mc{S}) \ni \alpha.\]
Since if $L \in \mc{C}$ then $\wf(L) \leq \sup(\scs(\mc{S}))$, we have $\sh(T) \leq \sh(\mc{S})$. So $\alpha \leq \sh(T') < \omega_1$. Since $\alpha$ was arbitrary, $\sh(\mc{L}^{\comp}_{\omega_1 \omega}) \geq \delta^1_2$. This proves the theorem.
\end{proof}

\section{Possible Scott Spectra of Theories}\label{sec:scs}
In this section, we will prove Theorems \ref{thm:main-result} and \ref{thm:main-result2} which completely classify the possible Scott spectra under the assumption of Projective Determinacy. We begin by going as far as we can without any assumptions beyond ZFC, and then we assume Projective Determinacy in order to get a cone of sets $X$ where the Scott spectrum contains either only $\omega_1^X$ for all $X$ on the cone, or only $\omega_1^X + 1$ for all $X$ on the cone, or both for all $X$ on the cone.

The following result is well-known.

\begin{lemma}\label{lem:models-admissible}
Let $T$ be an $\mc{L}_{\omega_1 \omega}$-sentence. If the Scott spectrum of $T$ is unbounded below $\omega_1$, then there is a set $Y$ such that for every $X \geq_T Y$, there is a model $\mc{A}$ of $T$ with $SR(\mc{A}) \geq \omega_1^{\mc{A},X} = \omega_1^X$.	In particular, $SR(\mc{A}) = \omega_1^X$ or $SR(\mc{A}) = \omega_1^X + 1$.
\end{lemma}
\begin{proof}
Choose $Y$ such that $T$ is $Y$-computable. This is a well-known application of Gandy's basis theorem; see \cite[Lemma 3.4]{Montalban13}.
\end{proof}

In the next lemma, we consider the linear orders which support back-and-forth relations on models of an $\mc{L}_{\omega_1 \omega}$-sentence $T$. This will give a $\bfSigma^1_1$ class of linear orders, which if it is unbounded will contain non-well-founded members (supporting non-standard back-and-forth relations).

\begin{theorem}\label{thm:second-dir}
Let $T$ be an $\mc{L}_{\omega_1 \omega}$-sentence. There is a $\bfSigma^1_1$ class of linear orders $\mc{C}$ such that
\[ \scs(T) = \mc{C} \cap \On. \]
If $\scs(T)$ is bounded below $\omega_1$, then $\mc{C} \subseteq \On$. Otherwise, if $\scs(T)$ is unbounded below $\omega_1$, then there is a set $Y$ such that
\begin{enumerate}
	\item for all $X \geq_T Y$, at least one of $\omega_1^X$ or $\omega_1^X + 1$ is in $\mc{C}$, and
	\item the only non-well-founded members of $\mc{C}$ are Harrison linear orders relative to some $X \geq_T Y$, i.e.\ $\omega_1^Y \cdot (1 + \mathbb{Q}) + \beta$ for some $Y$-computable $\beta$.
\end{enumerate}
\end{theorem}
\begin{proof}
If $\scs(T)$ is bounded below $\omega_1$, then we can just take $\mc{C} = \scs(T)$; this is a $\bfSigma^1_1$ class. So suppose that $\scs(T)$ is unbounded below $\omega_1$.

By Lemma \ref{lem:models-admissible}, there is a set $Y$ such that for every $X \geq_T Y$, there is a model $\mc{A}$ of $T$ with $SR(\mc{A}) = \omega_1^X$ or $SR(\mc{A}) = \omega_1^X + 1$.

The proof of the theorem is similar to the proof of Theorem \ref{thm:sh}. Let $\mc{C}$ be the $\bfSigma^1_1$ class of linear orders defined as follows. $(L,\leq)$ is in $\mc{C}$ if and only if there are:
\begin{enumerate}
	\item a model $\mc{A}$ of $T$,
	\item a set $X \geq_T Y$,
	\item a Harrison linear order $\mc{H}$ relative to $X$,
	\item an embedding $f \colon L \to \mc{H}$ such that $f(L)$ is an initial segment of $\mc{H}$, and
	\item non-standard back-and-forth relations $\precsim_\alpha$ on $\mc{A}$ indexed by $L$ (in the language $\mc{L}$ of $\mc{S}$, not of $T$), such that:
\begin{enumerate}
	\item for all $\alpha \in L$, there is $\bar{x}$ which is $\alpha$-free, i.e. for all $\bar{y}$ and $\beta < \alpha$, there are $\bar{x}'$ and $\bar{y}'$ such that $\bar{x}, \bar{y} \precsim_\beta \bar{x}', \bar{y}'$ and $\bar{x}' \nprecsim_\alpha \bar{x}$,
	\item the set of partial maps
	\[ \{ \bar{a} \mapsto \bar{b} : \bar{a} \precsim_{\alpha} \bar{b} \text{ for all $\alpha$}\} \]
	has the back-and-forth property.
\end{enumerate}
\end{enumerate}
Note that while in Theorem \ref{thm:sh} $\mc{C}$ was a (lightface) $\Sigma^1_1$ class, now $\mc{C}$ is $\Sigma^1_1(Y,T)$. We still prove, in the same way, that $\mc{C} \cap \On = \scs(T)$.

Now if $(L,\leq)$ is not well-founded, then for some $X \geq_T Y$, $L$ embeds into the Harrison linear order $\mc{H}$ relative to $X$. Since $L$ is not well-founded, it is itself isomorphic to $\omega_1^X \cdot (1 + \mathbb{Q}) + \beta$ for some $X$-computable ordinal $\beta$. By choice of $Y$, either $\omega_1^X$ or $\omega_1^X + 1$ is the Scott spectrum of a model of $T$.
\end{proof}

Our use of projective determinacy will be to have a uniform choice of whether $\omega_1^X$ or $\omega_1^X + 1$ (or both) is in the Scott spectrum; that is, we will be able to choose $Y$ such that for all $X \geq_T Y$, $\omega_1^X$ is in the Scott spectrum, or such that for all $X \geq_T Y$, $\omega_1^X + 1$ is in the Scott spectrum, or both. We now prove Theorem \ref{thm:main-result2}, which said that each Scott spectrum is built from a $\bfSigma^1_1$ class by either taking the well-founded part, the well-founded collapse, or both.

\begin{proof}[Proof of Theorem \ref{thm:main-result2}]
Theorems \ref{thm:construction} and \ref{thm:construction2} show that each of these is a Scott spectrum. In the other direction, Theorem \ref{thm:second-dir} says that each Scott spectrum is $\mc{C} \cap \On$ for some $\bfSigma^1_1$ class $\mc{C}$. Either $\mc{C} \subseteq \On$ (in which case $\mc{C}$ is bounded below $\omega_1$, and we are done) or there is a set $Y$ such that
\begin{enumerate}
	\item for all $X \geq_T Y$, $\omega_1^X$ or $\omega_1^X + 1$ is in $\mc{C}$, and
	\item the only non-well-founded members of $\mc{C}$ are Harrison linear orders relative to some $X \geq_T Y$.
\end{enumerate}

By projective determinacy, there is $Z \geq_T Y$ such that one of the following is true for all $X \geq_T Z$: 
\begin{enumerate}
	\item $\omega_1^X \in \mc{C}$ and $\omega_1^X + 1 \notin \mc{C}$,
	\item $\omega_1^X \notin \mc{C}$ and $\omega_1^X + 1 \in \mc{C}$, or
	\item $\omega_1^X \in \mc{C}$ and $\omega_1^X + 1 \in \mc{C}$.
\end{enumerate}
Let $\mc{C}'$ be the set of linear orders in $\mc{C}$ which embed into an initial segment of $\mc{H}^X$ for some $X \geq_T Z$. Then $\mc{C} \cap \On = \mc{C}' \cap \On$, and depending on which case we were in above, we have:
\begin{enumerate}
	\item $\mc{C} \cap \On = \mc{C}' \cap \On = \wf(\mc{C}')$,
	\item $\mc{C} \cap \On = \mc{C}' \cap \On = \wfc(\mc{C}')$, or
	\item $\mc{C} \cap \On = \mc{C}' \cap \On = \wf(\mc{C}') \cup \wfc(\mc{C}')$. \qedhere
\end{enumerate}
\end{proof}

We now give the proof of our alternate characterization.

\begin{proof}[Proof of Theorem \ref{thm:main-result}]
The Scott spectra which are bounded below $\omega_1$ clearly correspond to $\bfSigma^1_1$ sets of ordinals which are bounded below $\omega_1$. So it remains only to deal with the unbounded case.

First we show that if $\mc{C}$ is a $\bfSigma^1_1$ set of linear orders with either $\mc{C} \cap \On$ or $\{ \alpha : \alpha + 1 \in \mc{C} \cap \On\}$ stationary, then $\mc{C} \cap On$ is the Scott spectrum of a theory. Note that for each set $Y$, $\{ \omega_1^X : X \geq_T Y \}$ contains a club (see Remark \ref{rem:club}), and hence intersects either $\mc{C} \cap \On$ or $\{ \alpha : \alpha + 1 \in \mc{C} \cap On\}$. Thus exactly one of the following is true for cofinally many $X$ in the Turing degrees: $\omega_1^X$ is in $\mc{C} \cap \On$ (but $\omega_1^X + 1$ is not), or $\omega_1^X+1$ is in $\mc{C} \cap \On$ (but $\omega_1^X$ is not), or both are in $\mc{C} \cap \On$. By Projective Determinacy, one of these is true on a cone, say the cone above a set $Y$.

Now let $\mc{C}'$ be the $\bfSigma^1_1$ set of linear orders $(L,\leq)$ in $\mc{C}$ such that for some $X \geq_T Y$, $L$ embeds into an initial segment of $\mc{H}^X$. Then $\mc{C}' \cap \On = \mc{C} \cap \On$, and either:
\begin{enumerate}
	\item whenever $(L,\leq) \in \mc{C}'$ is not well-founded, $\wf(L)$ is in $\mc{C}' \cap \On$,
	\item whenever $(L,\leq) \in \mc{C}'$ is not well-founded, $\wf(L) + 1$ is in $\mc{C}' \cap \On$, or
	\item whenever $(L,\leq) \in \mc{C}'$ is not well-founded, $\wf(L)$ and $\wf(L) + 1$ are in $\mc{C}' \cap \On$.
\end{enumerate}
By the proof of Theorem \ref{thm:main-result2}, $\mc{C}' \cap \On$ is a Scott spectrum.

On the other hand, if $T$ is an $\mc{L}_{\omega_1 \omega}$-sentence whose Scott spectrum is unbounded below $\omega_1$, then by Theorem \ref{thm:main-result2} there is a set $Y$ such that for all $X \geq_T Y$, either $\omega_1^X$ or $\omega_1^X + 1$ is in $\scs(T)$. By projective determinacy, there is $Z \geq_T Y$ such that either for all $X \geq_T Z$, $\omega_1^X \in \scs(T)$, or for all $X \geq_T Z$, $\omega_1^X + 1 \in \scs(T)$. In Remark \ref{rem:club}, we noted that $\{ \omega_1^X : X \geq_T Z \}$ is stationary. This completes the proof.
\end{proof}

\begin{remark}
In Theorems \ref{thm:construction} and \ref{thm:construction2} we can get $T$ to be $\Pi^0_2$. Thus Theorem \ref{thm:all-pi-two} follows from Theorem \ref{thm:main-result2}.
\end{remark}

\begin{remark}
Note that the proofs of Lemma \ref{lem:models-admissible} and Theorems \ref{thm:second-dir}, \ref{thm:main-result}, and \ref{thm:main-result2} go through if we replace $T$ by a $PC_{\mc{L}_{\omega_1 \omega}}$ class of structures. Thus, under projective determinacy, the Scott spectra of $PC_{\mc{L}_{\omega_1 \omega}}$-classes are the same as the Scott spectra of $\mc{L}_{\omega_1 \omega}$ sentences. Thus we have proved Theorem \ref{thm:main-result-pseudo}.
\end{remark}

We can use the classification to find some interesting examples of Scott spectra.

\begin{proposition}
The following are all Scott spectra of $\mc{L}_{\omega_1 \omega}$-sentences:
\begin{enumerate}
	\item $\{ \alpha + 1 : \alpha < \omega_1 \}$.
	\item $\{ \alpha : \alpha < \omega_1 \text{ is an admissible ordinal}\}$.
	\item $\{ \alpha + 1 : \alpha < \omega_1 \text{ is an admissible ordinal}\}$.
\end{enumerate}
\end{proposition}
\begin{proof}
Note that if $(L,\leq)$ is not well-founded, then $\wfc(L)$ is a successor ordinal.  If $\mc{C}$ is the $\mc{L}_{\omega_1 \omega}$-definable class of all linear orders with an initial element and containing a dense interval (with endpoints), $\mc{C}$ contains no well-founded orders. Moreover, for each ordinal $\alpha < \omega_1$, $\alpha + \mathbb{Q} \in \mc{C}$ and $\wfc(\alpha + \mathbb{Q}) = \alpha + 1$. Thus
\[ \{ \wfc(L) : L \in \mc{C} \} = \{ \alpha : \alpha < \omega \text{ is a successor ordinal}\}\]
is a Scott spectrum.

In Subsection \ref{sub:harrison}, we remarked that class of Harrison linear orders is a $\Sigma^1_1$ class. Thus
\[ \{\wf(L) : L \in \mc{C} \} = \{ \omega_1^X : X \subseteq \omega \}\] 
and
\[ \{\wfc(L) : L \in \mc{C} \} = \{ \omega_1^X + 1 : X \subseteq \omega \}\] 
are Scott spectra. (Note that $\omega$ is an admissible which is not in the first spectrum, but we can easily add it in via Proposition \ref{prop:ctble-union}.)
\end{proof}

\section{Open Questions}

We begin by asking whether one can remove the assumption of Projective Determinacy in the classification of Scott spectra.

\begin{openquestion}
Classify the Scott spectra of $\mc{L}_{\omega_1 \omega}$-sentences in ZFC.
\end{openquestion}

We would also like to know a lightface classification. The proofs of Theorems \ref{thm:main-result} and \ref{thm:main-result2} do not go through for computable sentences because of the use of Projective Determinacy.

\begin{openquestion}
Classify the Scott spectra of computable $\mc{L}_{\omega_1 \omega}$-sentences.
\end{openquestion}

Finally, our construction relied upon being able to take infinite disjunctions, such as when we named each element of the order sort by a constant. A first-order theory cannot name each element of an infinite sort by a constant. Can our results be expanded to first-order theories?

\begin{openquestion}
Classify the Scott spectra of first-order theories.
\end{openquestion}

\bibliographystyle{alpha}
\bibliography{References}

\end{document}